\documentclass[12pt, reqno]{amsart}
\usepackage{amsthm}
\usepackage{amsfonts}
\usepackage{xr}
\usepackage{graphicx}
\usepackage{amsmath}
\usepackage{epsfig,subfigure}
\usepackage{color}
\usepackage{amssymb}

\def\Z{\mathbb{Z}}
\def\R{\mathbb{R}}
\def\N{\mathbb{N}}

\def\epsilon{\varepsilon}
\def\hat{\widehat}
\def\tilde{\widetilde}

\newcommand{\SE}{\setcounter{equation}{0} \section}
\newcommand{\be}{\begin{equation}}
\newcommand{\ee}{\end{equation}}
\newcommand{\baa}{\begin{array}}
\newcommand{\eaa}{\end{array}}
\newcommand{\ba}{\begin{eqnarray}}
\newcommand{\ea}{\end{eqnarray}}
\numberwithin{equation}{section}

\newtheorem{theo}{\bf Theorem}[section]
\newtheorem{lem}[theo]{\bf Lemma}
\newtheorem{pro}[theo]{\bf Proposition}

\newtheorem{rem}[theo]{\bf Remark}

\begin{document}
\date{\today}
\title[Spreading in space-time periodic media, Part 2]{Spreading in space-time periodic media governed by\\ a monostable  equation with free boundaries,\\ Part 2:  Spreading speed} 
\thanks{This work was supported by the Australian Research Council, the National Natural Science Foundation of China (11171319, 11371117) and the Fundamental Research Funds for the Central Universities. }
\author[w. Ding, Y. Du and X. Liang]{Weiwei Ding$^\dag$, Yihong Du$^\dag$ and Xing Liang$^\ddag$}
\thanks{
$^\dag$ School of Science and Technology, University of New England, Armidale, NSW 2351, Australia}
\thanks{$^\ddag$ School of Mathematical Sciences, University of Science and Technology of China, Hefei, Anhui, 230026, P.R. China}

\keywords{free boundary, space-time periodic media, spreading speed}

\begin{abstract}
This is Part 2 of our work aimed at classifying the long-time behavior of the solution to a free boundary problem with 
monostable reaction term in space-time periodic media. In Part 1 (see \cite{ddl}) we have established a theory on 
the existence and uniqueness of solutions to this free boundary problem with continuous initial functions, as well as
a spreading-vanishing dichotomy. 
We are now able to develop  the methods of Weinberger \cite{w1, w2} and others \cite{fyz,lyz,lz1,lz2,lui} to prove the 
existence of asymptotic spreading speed when spreading happens,  without knowing a priori the existence of the 
corresponding semi-wave solutions of the free boundary problem.  This is a completely different approach from earlier 
works on the free boundary model, where the spreading speed is determined by firstly showing the existence of a 
corresponding semi-wave. Such a semi-wave appears difficult to obtain by the earlier approaches in the case of space-time periodic media 
considered in our work here.
\end{abstract}

\maketitle


\SE{Introduction and main results}\label{sec1}
This is Part 2 of our work aimed at classifying the long-time dynamical behavior to a class of  space-time periodic reaction-diffusion equations with free boundaries of the form
\begin{equation}\label{eqf}\left\{\baa{ll}
u_t=du_{xx}+f(t,x,u),& g(t)<x<h(t),\; t>0,\vspace{3pt}\\
u(t,g(t))=u(t,h(t))=0,&t>0,\vspace{3pt}\\
g'(t)=-\mu u_x(t,g(t)),&t>0, \vspace{3pt}\\
h'(t)=-\mu u_x(t,h(t)),&t>0, \vspace{3pt}\\
g(0)=g_0,\; h(0)=h_0,\; u(0,x)=u_0(x),& g_0\leq x\leq h_0,\eaa\right.
\end{equation} 
where $x=g(t)$ and $x=h(t)$ are the moving boundaries to be determined together with $u(t,x)$, and $d$, $\mu$ are given positive constants.

The initial function $u_0$ belongs to $\mathcal{H}(g_0,h_0)$ for some $g_0<h_0$, where
\begin{equation*}
\mathcal{H}(g_0, h_0):=\Big\{\phi\in C([g_0,h_0]):\, \phi(g_0)=\phi(h_0)=0, \,\phi(x)>0 \hbox{ in }(g_0,h_0) \Big\}.
\end{equation*}

 The reaction term $f\!:\R^2\times\R^+\mapsto\R$ is continuous, of class $C^{\alpha/2,\alpha}(\R^2)$ in $(t,x)\in\R^2$ locally uniformly in $u\in\R^+$(with $0<\alpha<1$), and of class $C^{1}$ in $u\in\R^+$ uniformly in $(t,x)\in\R^2$. The basic assumptions on $f$ are:
\begin{equation}\label{zero}
f(t,x,0)=0 \quad \hbox{for all  } (t,x)\in\R^2,
\end{equation}
there exists
  $M>0$ such that 
\begin{equation}\label{hyp2}
f(t,x,u)\leq 0\,\hbox{ for all }\, (t,x)\in\R^2,\, u\geq M,
\end{equation}
and $f$ is  $\omega$-periodic in $t$ and $L$-periodic in $x$ for some positive constants $\omega$ and $L$, that is, 
\begin{equation}\label{period}
\left\{
\begin{array}{l}
f(t+\omega,x,u)=f(t,x,u)\\
 f(t,x+L,u)=f(t,x,u)
\end{array}\right. \hbox{ for all }\, (t,x)\in \R^2,\,u\geq 0.
\end{equation}

 In this work, we regard \eqref{eqf} as describing 
the spreading of a new or invasive species over a one-dimensional habitat, where  $u(t,x)$ represents the population density of 
the species at location $x$ and time $t$, the reaction term $f$ measures the growth rate, the free boundaries $x=g(t)$ and 
$x=h(t)$ stand for the edges of the expanding population range, namely the spreading fronts. The Stefan conditions $g'(t)=-\mu u_x(t,g(t))$ and $h'(t)=-\mu u_x(t,h(t))$ may be 
interpreted as saying that  the spreading front expands at a speed  proportional to the population gradient at the front;
a deduction of these conditions from ecological considerations can be found in \cite{BDK}. When 
$f(t,x,u)$ is periodic with respect to $x$ and $t$ as described in \eqref{period}, problem \eqref{eqf} represents spreading of the species in a  heterogeneous environment that 
is periodic in both space and time. 

 In the special case that the function $f$ does not depend on $x$ and $t$, and is of logistic type, that is, 
$$f(u)=u(a-bu)\hbox{ for some positive constants } a \hbox{ and } b,$$ 
such a problem was first studied in \cite{dlin} for the spreading of a new or invasive species. It is proved that, when
$$u_0\in C^2([g_0,h_0]),\, u_0(g_0)=u_0(h_0)=0, \,u_0(x)>0 \hbox{ in }(g_0,h_0), $$
there exists a unique solution $(u,g,h)$ with $u(t,x)>0$, $g'(t)<0$ and $h'(t)>0$ for all $t>0$ and $g(t)<x<h(t)$, 
and a spreading-vanishing dichotomy holds, namely, there is a barrier $R^*$ on the size of the population range, such that either 
\begin{itemize}
\item[(i)] {\bf Spreading}: the population range breaks the barrier at some finite time (i.e., $h(t_0)-g(t_0)\geq R^*$ for some $t_0\geq 0$), and then the free boundaries go to infinity as $t\to\infty$ (i.e., $\lim_{t\to\infty}h(t)=\infty$ and $\lim_{t\to\infty}g(t)=-\infty$), and the population spreads to the entire space and stabilizes at its positive steady state (i.e. $\lim_{t\to\infty}u(t,x)=a/b $ locally uniformly in $x\in\R$) or   

\item[(ii)] {\bf Vanishing}: the population range never breaks the barrier (i.e. $h(t)-g(t)< R^*$ for all $t>0$), 
and the population vanishes (i.e. $\lim_{t\to\infty}u(t,x)=0$). 
\end{itemize}

Moreover, when spreading occurs, the asymptotic spreading speed can be determined, i.e., 
$$\lim_{t\to\infty}-g(t)/t=\lim_{t\to\infty}h(t)/t=c,
$$
 where $c$ is the unique positive constant such that  the problem
\begin{equation*}
\left\{\baa{l}
dq_{xx}-cq_x+q(a-bq)=0, \;q(x)>0 \quad\hbox{for } x\in (0,\infty),\vspace{3pt}\\
q(0)=0, \quad \mu q_x(0)=c,\quad q(\infty)=1 
\eaa\right.
\end{equation*} 
has a (unique) solution $q$. Such a solution $q(x)$ is called a semi-wave with speed $c$.

These results have subsequently been extended to more general situations in several directions. But as we mentioned in the 
Introduction of Part 1 (\cite{ddl}), in all the previous works on this problem, the spreading speed is determined by the corresponding semi-wave solution which, in our current space-time periodic case, appears difficult to establish by adapting the existing approaches.

In this paper we use a different approach to determine the spreading speed for the space-time periodic case of problem \eqref{eqf} with a  monostable $f$. This approach is based on recent developments of Weinberger's ideas first appeared in \cite{w1}, where the existence of spreading speed for the corresponding Cauchy problem is proved without knowing the existence of the corresponding traveling wave solutions. 
In \cite{w1}, Weinberger
 established the existence of spreading speed for a scalar discrete-time recursion with a translation-invariant order-preserving monostable operator. Such a method was generalized in \cite{lui} to systems of discrete-time recursions, and then in \cite{w2} to scalar discrete-time recursions in spatially periodic habitats. The theory in \cite{lui,w1} was further developed in \cite{lz1} to the investigation of both discrete and continuous semiflows with a monostable structure, and then was extended to time-periodic semiflows in \cite{lyz}, to space-periodic semiflows in \cite{lz2}, and recently to space-time periodic semiflows in \cite{fyz}.

However, to adapt these ideas to treat our free boundary problem here, it is necessary to firstly extend the existence 
and uniqueness theory for \eqref{eqf} with $C^2$ initial functions (see \cite{dlin}) to the case that the initial functions are 
merely continuous, which has not been considered before and requires  new techniques. This step has now been carried out in Part 1 of this work.
Moreover, in Part 1, we have also proved the continuous dependence of the solution to the initial function, and established some comparison principles and a 
 spreading-vanishing dichotomy for \eqref{eqf}.  

With these preparations, we are now able to establish the existence of asymptotic spreading speed for \eqref{eqf}, by further developing the techniques of Weinberger \cite{w1, w2} and several other recent works \cite{fyz, lyz, lz1, lz2, lui}. To do this, we assume that the associated nonlinear term $f$ admits a {\it monostable} structure, characterized by the following assumption (H). 

\smallskip

\noindent{\bf Assumption (H):}

\begin{itemize}
\item[(i)]
{\it  The following problem
\begin{equation}\label{psteady11}
\left\{\baa{l}
p_t=dp_{xx}+f(t,x,p)\,\, \hbox{ in }\,(t,x)\in\R^2,\vspace{3pt}\\
p(t,x)\,\hbox{ is $\omega$-periodic in $t$ and $L$-periodic in $x$},\eaa\right.
\end{equation}   

admits a unique positive solution $p(t,x)\in C^{1,2}(\R^2)$;}
\item[(ii)] {\it for any $v_0\in C(\R)\cap L^\infty(\R)$ with
$\inf_{x\in\R} v_0(x)>0$, there holds
\begin{equation}\label{monoconver}
v(t+s,x;v_0) - p(t+s,x)\to 0\,\hbox{ as }\, s\to\infty  
\end{equation}
 uniformly in $ (t,x)\in [0,\infty)\times \R$,
where $v(t,x;v_0)$ is the solution of the Cauchy problem 
\begin{equation}\label{cauchy}
\left\{\baa{ll}
v_t=dv_{xx}+f(t,x,v),& x\in\R,\,t>0,\vspace{3pt}\\
v(0,x)=v_0(x),& \,x\in\R.\eaa\right.
\end{equation} 
}
\end{itemize}

Under the assumption (H), it is recently proved in \cite{fyz} that, the Cauchy problem \eqref{cauchy} has a rightward spreading speed $\bar{c}^*_+$ and a leftward spreading speed $\bar{c}^*_-$. More precisely, for any nonnegative non-null compactly supported initial datum $v_0$ with $v_0\leq p(0,x)$ for $x\in\R$, there holds 
\begin{equation*}
\left\{\baa{ll}
\lim_{t\to\infty} \sup_{x\in [-c_2t, c_1t]} \big|v(t,x,v_0)-p(t,x)\big|=0 & \hbox{ when }\,-\bar{c}^*_{-}< -c_2<c_1<\bar{c}^*_{+},\vspace{3pt}\\
\lim_{t\to\infty}\sup_{x\in (-\infty,-c'_2]\cup [c'_1t,+\infty)} v(t,x;v_0)=0 & \hbox{ when }\,c_2'>\bar{c}^*_{-}\hbox{ and }c_1'>\bar{c}^*_{+},\eaa\right.
\end{equation*}
where $v(t,x,v_0)$ is the unique solution of \eqref{cauchy}. In this current work, we show that under the same conditions, whenever spreading occurs,  the free boundary problem \eqref{eqf} also has a leftward and a rightward spreading speed.  More precisely, we have the following theorem.

\begin{theo}\label{spreadspeed}
Suppose that \eqref{zero}, \eqref{hyp2}, \eqref{period} and {\rm (H)} are satisfied.  Then there exist constants $c^*_{-,\mu}>0$ and  $c^*_{+,\mu}>0$ such that for any given $u_0\in\mathcal{H}(g_0,h_0)$ with $u_0(x)\leq p(0,x)$ in $\R$ such that 
$\lim_{t\to\infty}h(t)=\lim_{t\to\infty}-g(t)=\infty$ and 
\begin{equation}\label{assupcon}
\lim_{t\to \infty} \big |u(t,x)-p(t,x)\big |=0\,\hbox{ locally uniformly in }\,x\in\R, 
\end{equation}
the following conclusions hold:
\begin{equation}\label{spreadu}
\lim_{t\to\infty}\sup_{-c_2t\leq x\leq c_1t} \big|u(t,x)-p(t,x)\big|=0\,\hbox{ when }\, -c^*_{-,\mu}< -c_2<c_1<c^*_{+,\mu},
\end{equation} 
 and 
\begin{equation}\label{spreadgf}
\lim_{t\to\infty} \frac{g(t)}{t}=-c^*_{-,\mu},\qquad \lim_{t\to\infty} \frac{h(t)}{t}=c^*_{+,\mu}.
\end{equation}
Here $(u,g,h)$ is the solution to \eqref{eqf} with initial datum $u_0$ and $p$ is the unique positive solution of \eqref{psteady11}.
\end{theo}

The above theorem indicates that $c^*_{+,\mu}$ (resp. $c^*_{-,\mu}$) is the rightward (resp. leftward) spreading speed for problem \eqref{eqf}.

\begin{rem}\label{rmk1} The restriction $u_0(x)\leq p(0,x)$ in Theorem \ref{spreadspeed} is rather unnatural. We will show that it can be removed under mild additional assumptions on $f$ near $u=p(t,x)$; see Section 2.1 below for details.
\end{rem}
 
The proof of Theorem~\ref{spreadspeed} is based on ideas in \cite{fyz,lyz,lz1,lz2,w2}, but considerable technical changes are needed since the introduction of the free boundary here. As a consequence, the proof of Theorem~\ref{spreadspeed} is rather involved.

We now give some examples of nonlinearities $f$ for which the hypothesis (H) can be easily checked. The simplest example is the logistic nonlinearity 
\begin{equation}\label{logic}
f(t,x,u)=u\big(a(t,x)-b(t,x)u\big)
\end{equation}
where $a,\,b$ are of class $C^{\alpha/2,\alpha}$ which are $\omega$-periodic in $t$ and $L$-periodic in $x$, and there are positive constants $\kappa_1$, $\kappa_2$ such that 
$\kappa_1\leq a(t,x)\leq \kappa_2$ and $\kappa_1\leq b(t,x)\leq \kappa_2$ for all $(t,x)\in\R^2$.  It is well known that, with such a nonlinearity $f$, problem \eqref{psteady11} admits a unique positive solution $p(t,x)\in C^{1,2}(\R^2)$, \eqref{monoconver} holds and for any nonnegative bounded non-null initial function $v_0\in C(\R)$, there holds 
\begin{equation*}
v(t+s,x;v_0) - p(t+s,x)\to 0\,\hbox{ as }\, s\to\infty  \, \hbox{ locally uniformly in }\, (t,x)\in\R^2,
\end{equation*}
where $v(t,x;v_0)$ is the unique solution of the  Cauchy problem \eqref{cauchy}. In fact, these existence, uniqueness and stability results hold for more general $f$ satisfying (in addition to the basic assumptions \eqref{zero}, \eqref{hyp2} and \eqref{period}),
\begin{equation}\label{hyp1}
\forall\, (t, x) \in\R^2,\hbox{ the function }\, u\mapsto f(t,x,u)/u \,\hbox{ is decreasing for }\, u>0,
\end{equation}
and the generalized principal eigenvalue of the linearized problem (at $u=0$) is negative (see \cite{na1,na3}).

An example satisfying (H) but not \eqref{hyp1} is 
\begin{equation}\label{dege}
f(t,x,u)=a(t,x)u^k(1-u)\,\hbox{ for some }\, k>1,
\end{equation}
where $a(t,x)$ is a positive function of class $C^{\alpha/2,\alpha}$, and is $\omega$-periodic in $t$ and $L$-periodic in $x$.  It follows from \cite[Proposition 1.7]{na2} that $p(t,x)\equiv 1$ is the unique positive solution for problem \eqref{psteady11} with nonlinearity \eqref{dege}. A simple comparison argument involving a suitable ODE problem shows that  \eqref{monoconver} holds for such a nonlinearity. 

Regarding sufficient conditions for spreading to happen for \eqref{eqf},  when $f$ is of type \eqref{logic},  the spreading-vanishing dichotomy proved in Part 1 shows that there exists a positive constant $R$ (independent of $u_0$) such that $h_0-g_0\geq R$ implies  spreading. When $f$ is of type \eqref{dege}, sufficient conditions for spreading can be found in
\cite[Theorem 1.1 and Remark 2.4]{sun}.

Finally, let us consider the behavior of the spreading speeds for problem \eqref{eqf} as $\mu$ increases to $\infty$. We have the following theorem. 

\begin{theo}\label{limitmu}
Let $c^*_{\pm,\mu}$ be the spreading speeds obtained in Theorem~\ref{spreadspeed}. Then $c^*_{\pm,\mu}$ are nondecreasing in $\mu>0$, and 
\begin{equation*}
\lim_{\mu\to\infty} c^*_{-,\mu}=\bar{c}^*_- \quad\hbox{and}\quad \lim_{\mu\to\infty} c^*_{+,\mu}=\bar{c}^*_+,
\end{equation*}
where $\bar{c}^*_+$ $(\mbox{resp. } \bar{c}^*_-)$ is the rightward $($resp. leftward$)$ spreading speed for problem \eqref{cauchy}.
\end{theo}

The general strategy in proving  Theorems~\ref{spreadspeed} and \ref{limitmu} can be summarised as follows. We will first show the existence of rightward spreading speed for the following free boundary problem 
\begin{equation}\label{eqfunbd}\left\{\baa{ll}
u_t=du_{xx}+f(t,x,u),& -\infty<x<h(t),\quad t>0,\vspace{3pt}\\
u(t,h(t))=0,\,\,\, h'(t)=-\mu u_x(t,h(t)),&t>0, \vspace{3pt}\\
h(0)=h_0,\quad u(0,x)=u_0(x),& -\infty< x\leq h_0,\eaa\right.
\end{equation} 
with initial data $u_0\in {\mathcal{H}_+}(h_0)$, where
\begin{equation*}
{\mathcal{H}_+}(h_0):=\Big\{\phi\in C\big((-\infty,h_0]\big)\cap L^{\infty}\big((-\infty,h_0]\big):\,\phi(h_0)=0, \,\phi(x)>0 \hbox{ in }(-\infty,h_0) \Big\}.
\end{equation*} 
Then we will prove that this speed is indeed the rightward spreading speed for problem \eqref{eqf}, and that it converges to the rightward spreading speed for the Cauchy problem \eqref{cauchy} as $\mu\to\infty$. Similarly, to obtain the existence and convergence of leftward spreading speed for \eqref{eqf}, it suffices to prove these for the problem 
\begin{equation}\label{eqfunbdn}\left\{\baa{ll}
u_t=du_{xx}+f(t,x,u),& g(t)<x<\infty,\quad t>0,\vspace{3pt}\\
u(t,g(t))=0,\,\,\, g'(t)=-\mu u_x(t,g(t)),&t>0, \vspace{3pt}\\
g(0)=g_0,\quad u(0,x)=u_0(x),& g_0 \leq x < \infty,\eaa\right.
\end{equation} 
with initial data $u_0\in {\mathcal{H}_-}(g_0)$, where
\begin{equation*}
{\mathcal{H}_-}(g_0):=\Big\{\phi\in C\big([g_0,\infty)\big)\cap L^{\infty}\big([g_0,\infty)\big):\,\phi(g_0)=0, \,\phi(x)>0 \hbox{ in }(g_0,\infty) \Big\}.
\end{equation*}

\smallskip
\noindent{\bf Outline of the paper:} The remaining part of this paper is organized as follows. In Section~\ref{sec2}, we first give a proof for the statement in Remark \ref{rmk1} and then introduce some notations and state some common properties of the solutions to problems \eqref{eqf}, \eqref{eqfunbd} and \eqref{eqfunbdn}. Section~\ref{sec3} is devoted to the proof for the existence of spreading speeds for problems \eqref{eqfunbd} and \eqref{eqfunbdn}. The proof of Theorem~\ref{spreadspeed} is given in Section~\ref{sec4} and  the proof of Theorem~\ref{limitmu} is carried out in Section~\ref{sec5}.

\section{Preliminaries}\label{sec2}
In this section, we  prove the statement in Remark \ref{rmk1}, and then introduce some notations and basic facts to be used in the subsequent sections.

\subsection{On the condition $u_0(x)\leq p(0,x)$ in Theorem \ref{spreadspeed}}

The statement in Remark \ref{rmk1} clearly follows from the result below.

\begin{pro}\label{u<p}
Suppose that $f$ satisfies \eqref{zero}, \eqref{hyp2}, \eqref{period} and {\rm (H)}.  Suppose further  there exists $\epsilon_0>0$ small such that for every $(t,x)\in\R^2$, we have $f(t,x,\cdot)\in C^2(I_{t,x})$ with $I_{t,x}:=[(1-\epsilon_0)p(t,x), (1+\epsilon_0)p(t,x)]$, and
\begin{equation}
\label{f-p}
\frac{f(t,x,u)}{u} \mbox{ is nonincreasing in $u$},\; f_{uu}(t,x,u)\leq 0\;  \mbox{ for } u\in I_{t,x}.
\end{equation}
Let $u_0\in\mathcal{H}(g_0,h_0)$ and $(u,g,h)$ be the unique solution of \eqref{eqf}. 
Then there exists $t_0>0$ such that
\[
u(t_0,x)<p(0,x) \mbox{ for } x\in\R.
\]
\end{pro}

Let us note that the functions $f$ given in \eqref{logic} and \eqref{dege} also satisfy \eqref{f-p}.

To prove Proposition \ref{u<p}, we will use the following result on the corresponding Cauchy problem \eqref{cauchy}, which may have independent interest.
\begin{pro}\label{v<p}
Suppose that $f$ satisfies all the assumptions in Proposition \ref{u<p}. Let 
 $v(t,x)$ be the unique solution of \eqref{cauchy} with initial function $v_0\in C(\R)$ nonnegative and having compact support. 
 Then
 there exist $t_0>0$ and $\delta_0>0$ such that
\begin{equation}\label{v<p-delta}
v(t_0,x)\leq p(0,x)-\delta_0 \mbox{ for } x\in\R.
\end{equation}
\end{pro}

\begin{proof}
Due to \eqref{f-p}, for $\epsilon\in (0,\epsilon_0]$, 
\[
[(1-\epsilon)p]_t-d[(1-\epsilon)p]_{xx}\leq f(t,x, (1-\epsilon)p) \mbox{ for } t,\, x\in\R.
\]
It follows that
\[
v_\epsilon(t,x):=\max\{v(t,x), (1-\epsilon)p(t,x)\}
\]
satisfies, in the weak sense,
\[
(v_\epsilon)_t-d(v_\epsilon)_{xx}\leq f(t,x, v_\epsilon) \mbox{ for } t>0,\, x\in\R.
\]

For clarity, we divide the argument below into several steps.
\smallskip

\noindent
{\bf Step 1.}
Define
\[
w_\epsilon(t,x):=v_\epsilon(t,x)-p(t,x).
\]
We show that for all large $t>0$, say $t\geq T_0$,
\begin{equation}
\label{w-ep}
(w_\epsilon)_t-d(w_\epsilon)_{xx}\leq f_u(t,x,p(t,x))w_\epsilon \mbox{ for all } x\in\R.
\end{equation}

Clearly
\[
(w_\epsilon)_t-d(w_\epsilon)_{xx}\leq f(t,x,v_\epsilon)-f(t,x,p).
\]
By \eqref{monoconver} and a simple comparison argument,  for all large $t$, say $t\geq T_0=T_0(\epsilon)$, $v(t,x)\leq (1+\epsilon)p(t,x)$ for $x\in\R$. It follows that
\[
(1-\epsilon)p(t,x)\leq v_\epsilon(t,x)\leq (1+\epsilon)p(t,x) \mbox{ for all } x\in\R,\; t\geq T_0.
\]
Hence by the Taylor expansion and \eqref{f-p} we obtain, for $t\geq T_0$ and $x\in\R$,
\begin{align*}
&f(t,x,v_\epsilon)-f(t,x,p)\\
&=f_u(t,x,p)(v_\epsilon-p)+\frac12f_{uu}(t,x,\theta_\epsilon)(v_\epsilon-p)^2\\
&\leq f_u(t,x,p)w_\epsilon 
\end{align*}
since
\[
\theta_\epsilon=\theta_\epsilon(t,x)\in [(1-\epsilon)p(t,x), (1+\epsilon)p(t,x)].
\]
This proves \eqref{w-ep}.

\smallskip

\noindent
{\bf Step 2.} Comparison via a linear equation.

In this step, we obtain an upper bound for $w_\epsilon$ by making use of \eqref{w-ep} and
 the following eigenvalue problem
\[
\left\{
\begin{array}{l}
\phi_t-d\phi_{xx}-f_u(t,x, p(t,x))\phi =\lambda\phi\; \mbox{ for } (t,x)\in\R^2,\smallskip\\
 \phi>0 \mbox{ and is $\omega$-periodic in $t$, $L$-periodic in $x$}.
\end{array}
\right.
\]

It is well known that this eigenvalue problem has an eigenpair $(\lambda, \phi)=(\lambda_1,\phi_1)$ (see \cite{na1}).
So we have
\[
(\phi_1)_t-d(\phi_1)_{xx}=a(t,x)\phi_1 \mbox{ for } t,\; x\in\R,
\]
where 
\[
a(t,x):=f_u(t,x,p(t,x))+\lambda_1.
\]

Set
\[
V(t,x):=e^{\lambda_1t}w_\epsilon(t,x).
\]
From \eqref{w-ep} we obtain
\[
V_t-dV_{xx}\leq a(t,x) V \mbox{ for } t\geq T_0,\; x\in\R.
\]

By our assumption on $f$, there exists $K>0$ such that
\[
f(t,x, v(t,x))\leq K v(t,x) \mbox{ for all } t>0,\; x\in\R.
\]
Hence
\[
v_t-dv_{xx}\leq Kv,\; v(0,x)=v_0(x).
\]
Since $v_0$ has compact support, this implies, by the heat kernel expression of $v$, that for every fixed $t>0$,
\[
v(t,x)\to 0 \mbox{ as } |x|\to\infty.
\]
(See Lemma 2.2 in \cite{dm} for a simple proof.)
Thus we can find $l_0>0$ large so that
\[
v(T_0,x)\leq \frac 12 (1-\epsilon)p(T_0,x) \mbox{ for } |x|\geq l_0.
\]
It follows that
\[
V(T_0,x)=-e^{\lambda_1 T_0}\epsilon p(T_0,x) \mbox{ for } |x|\geq l_0.
\]
Therefore we can find $\delta>0$ small such that
\[
V(T_0,x)+\delta \phi_1(T_0,x)<0 \mbox{ for } |x|\geq l_0.
\]

We now define
\[
W_0(x):=\max\{0, V(T_0,x)+\delta \phi_1(T_0,x)\}.
\]
Clearly $W_0(x)=0$ for $|x|\geq l_0$.  Moreover, 
\[
\tilde V(t,x):=V(t,x)+\delta \phi_1(t,x)
\]
satisfies
\[
\left\{
\begin{array}{ll}
\tilde V_t-d \tilde V^\epsilon_{xx}\leq a(t,x)\tilde V & \mbox{ for } t> T_0,\; x\in\R,\smallskip\\
\tilde V(T_0,x)\leq W_0(x) & \mbox{ for } x\in\R.
\end{array}
\right.
\]
Let $W(t,x)$ be the unique solution to
\[\left\{\begin{array}{ll}
W_t-dW_{xx}=a(t,x)W & \mbox{ for } t>T_0,\; x\in\R,\smallskip\\
W(T_0,x)=W_0(x) &\mbox{ for } x\in\R.
\end{array}
\right.
\]
Then clearly
\begin{equation}\label{tilde-V-W}
e^{\lambda_1t}w_\epsilon(t,x)+\delta\phi_1(t,x)=\tilde V(t,x)\leq W(t,x) \mbox{ for } t\geq T_0, \; x\in\R.
\end{equation}
This is the estimate for $w_\epsilon$ we wanted to obtain in this step.

\smallskip

\noindent
{\bf Step 3.}
We prove that
\[
\lim_{t\to\infty} W(t,x)=0 \mbox{ uniformly in } x\in\R.
\]

If $W_0(x)\equiv 0$, then $W(t,x)\equiv 0$ and there is nothing left to prove. So assume that $W_0\not\equiv 0$. Since $W_0\in L^\infty(\R)$, there exists $M_0>0$ such that
\[
W_0(x)\leq M_0 \phi_1(t,x) \mbox{ for } t,\; x\in\R.
\]
It follows that $W(t,x)\leq M_0\phi_1(t,x)$ for all $t>T_0$ and $x\in\R$. Since $W_0(x)$ has compact support, and
$a\in L^\infty(\R^2)$, as before we have
\[
\lim_{|x|\to\infty} W(t,x)=0 \mbox{ for any fixed } t>T_0.
\]
Therefore, for each $t>T_0$, there exists $M(t)>0$ and $x_t\in\R$ such that
\[
W(t,x)\leq M(t)\phi_1(t,x) \mbox{ for all } x\in\R,\; W(t,x_t)=M(t)\phi_1(t, x_t).
\]
$M(t)$ must be nonincreasing in $t$, since if $T_0<t_1<t_2$, then from $W(t_1,x)\leq M(t_1)\phi_1(t_1,x)$ and the comparison principle we deduce
\[
W(t,x)\leq M(t_1)\phi_1(t,x) \mbox{ for } t>t_1,\; x\in\R.
\]
Hence $M(t_2)\leq M(t_1)$. We may then define
\[M_\infty:=\lim_{t\to\infty} M(t).
\]
Clearly $M_\infty\geq 0$. If $M_\infty=0$ then it follows immediately that $\lim_{t\to\infty} W(t,x)=0$ uniformly in $x\in\R$, as required. 

If $M_\infty>0$, we are going to derive a contradiction. Choose an increasing sequence $\{t_n\}$ such that
$\lim_{n\to\infty} t_n=\infty$, and denote $x_n:=x_{t_n}$. So we have
\[
W(t_n,x_n)=M(t_n)\phi_1(t_n, x_n).
\]
Set
\[
W_n(t,x)=W(t_n+t, x_n+x),
\]
and write
\[
t_n=k_n\omega +\tilde t_n,\; x_n=l_nL+\tilde x_n \mbox{ with } k_n, l_n\in\N,\; \tilde t_n\in [0, \omega),
\; \tilde x_n\in [0,L).
\]
Then
\[
(W_n)_t-d(W_n)_{xx}=a(\tilde t_n+t, \tilde x_n+x) W_n.
\]
By passing to a subsequence we may assume that
\[
\lim_{n\to\infty}\tilde t_n=\tilde t\in [0,\omega],\; \lim_{n\to\infty} \tilde x_n=\tilde x\in [0,L].
\]
By standard parabolic estimates and a diagonal argument, we may also assume (by passing to a further subsequence) that $W_n\to W_\infty$ in $C_{loc}^{1,2}(\R\times \R)$. It follows that
\begin{equation}\label{W-infty}
(W_\infty)_t-d(W_\infty)_{xx}=a(\tilde t+t, \tilde x+x)W_\infty \mbox{ for } t,\; x\in\R.
\end{equation}
Since
\[
W_n(t,x)\leq M(t_n+t)\phi_1(t_n+t,x_n+x)=M(t_n+t)\phi_1(\tilde t_n+t,\tilde x_n+x)
\]
and
\[
W_n(0,0)=M(t_n)\phi_1(\tilde t_n,\tilde x_n),
\]
we further have
\[
W_\infty(t,x)\leq M_\infty \phi_1(\tilde t+t, \tilde x+x),\; W_\infty(0,0)=M_\infty \phi_1(\tilde t, \tilde x).
\]
As $M_\infty \phi_1(\tilde t+t,\tilde x+x)$ also solves \eqref{W-infty}, the strong maximum principle infers that
\[
W_\infty(t,x)\equiv M_\infty \phi_1(\tilde t+t, \tilde x+x).
\]

We now fix $n_0\in\N$ large such that
\[
n_0L>2l_0,
\]
and define, for $j,\, n\in\N$,
\[
W^j(t,x):=W(t, x+jn_0L),\; W^j_n(t,x):=W^j(t_n+t, x_n+x).
\]
Then
\begin{equation}
\label{W-j}
W^j_t-dW^j_{xx}=a(t,x+jn_0L)W^j=a(t,x)W^j
\end{equation}
and
\[
W_n(t,x)=W^j(t_n+t, x_n-jn_0L+x)=W^j_n(t, x-jn_0L).
\]
Hence, after passing to the same subsequence (independent of $j$),
\[
\lim_{n\to\infty}W_n^j(t,x-jn_0L)= W_\infty(t,x)\equiv M_\infty\phi_1(\tilde t+t, \tilde x+x) \mbox{ in } C^{1,2}_{loc}(\R\times \R).
\]
It follows that
\[
\lim_{n\to\infty} W_n^j(t,x)=M_\infty \phi_1(\tilde t+t, \tilde x+x+jn_0L)=M_\infty\phi_1(\tilde t+t,\tilde x+x),
\]
and for each $k\in\N$,
\[
\lim_{n\to\infty} \Sigma_{j=1}^k W^j_n(t,x)=kM_\infty\phi_1(\tilde t+t,\tilde x+x).
\]

On the other hand, we note that due to the choice of $n_0$, $\{W^j(T_0,x)\}_{j\in\N}$ is a sequence of nonnegative functions with disjoint supporting sets, and for each $k\in\N$, 
\begin{equation}
\label{W-j-initial}
\Sigma_{j=1}^kW^j(T_0,x)\leq M_0\phi_1(T_0,x)
\end{equation}
since
\[
W^j(T_0,x)\leq M_0\phi_1(T_0,x+jn_0L)=M_0\phi_1(T_0,x) \mbox{ for } x\in [-l_0-jn_0L, l_0-jn_0L]
\]
and
$W^j(T_0,x)=0$ for $x\not\in [-l_0-jn_0L, l_0-jn_0L]$.

By \eqref{W-j} we see that $\Sigma_{j=1}^kW^j(t,x)$ and $ M_0\phi_1(t,x)$ satisfy the same linear differential equation, and so, in view of \eqref{W-j-initial}, we can apply the comparison principle to deduce
\[
\Sigma_{j=1}^kW^j(t,x)\leq  M_0\phi_1(t,x)\leq M_0\|\phi_1\|_\infty \mbox{ for } t>T_0,\; x\in\R.
\]
It follows that
\[
\Sigma_{j=1}^kW_n^j(t,x)\leq  M_0\|\phi_1\|_\infty \mbox{ for } t>T_0-t_n,\; x\in\R.
\]
Letting $n\to\infty$ we thus obtain
\[
kM_\infty \phi_1(\tilde t+t, \tilde x+x)\leq M_0\|\phi_1\|_\infty \mbox{ for } t,\; x\in\R.
\]
Since $\sigma_0:=\min_{(t,x)\in\R^2}\phi_1(t,x)>0$, the above inequality leads to a contradiction if we choose $k\in\N$ large enough. This proves our claim that $M_\infty=0$. Thus we always have
\[
\lim_{t\to\infty} W(t,x)=0 \mbox{ uniformly for } x\in\R.
\]

\noindent
{\bf Step 4.} Completion of the proof.

Since $\phi_1(t,x)\geq \sigma_0$, we have
\[
V(t,x)+\delta\sigma_0\leq \tilde V(t,x)\leq W(t,x) \mbox{ for } t>T_0,\; x\in\R.
\]
Choose $T_1>T_0$ such that $W(t,x)\leq \frac12 \delta\sigma_0 $ for $x\in\R$ and $t\geq T_1$. Then
\[
V(t,x)\leq -\frac12 \delta\sigma_0<0 \mbox{ for } x\in\R,\; t\geq T_1.
\]
Therefore, for $t\geq T_1$,
\[
w_\epsilon(t,x)=e^{-\lambda_1t}V(t,x)\leq-\frac 12 \delta\sigma_0 e^{-\lambda_1 t}<0 \mbox{ for } x\in\R.
\]
It follows that
\[
v(t,x)\leq v_\epsilon(t,x)=w_\epsilon(t,x)+p(t,x)\leq p(t,x)-\frac 12 \delta\sigma_0 e^{-\lambda_1 t} \mbox{ for } t\geq T_1,\; x\in\R.
\]
Taking $k\in\N$ such that $t_0:=k\omega\geq T_1$, and denoting $\delta_0:=\frac 12 \delta\sigma_0 e^{-\lambda_1 t_0}$, we then obtain
\[
v(t_0,x)=v(k\omega, x)\leq p(k\omega, x)-\delta_0=p(0,x)-\delta_0 \mbox{ for } x\in\R.
\]
So \eqref{v<p-delta} holds and the proof is complete.
\end{proof}
\begin{proof}[Proof of Proposition \ref{u<p}] We use Proposition \ref{v<p} with $v_0=u_0$. By the comparison principle we have
\[
u(t,x)\leq v(t,x) \mbox{ for } t>0,\, x\in\R.
\]
By \eqref{v<p-delta} we thus have $u(t_0,x)\leq v(t_0,x)<p(0,x)$ for $x\in\R$.
\end{proof}

\subsection{Notations and basic facts} From now on, we always assume that 
\[
\mbox{$f$ satisfies \eqref{zero}, \eqref{hyp2}, \eqref{period} and {\rm (H)}.}
\]
For any $h_0\in\R$ and any $u_0\in\mathcal{H}_+(h_0)$, 
 the unique solution of equation \eqref{eqfunbd} with initial value $u_+(0,x) = u_0(x)$ in $(-\infty, h_0]$
 is denoted by 
 $\big(u_+(t,x;u_0),h_+(t;u_0)\big)$; for any $g_0\in\R$ and any $u_0\in\mathcal{H}_-(g_0)$,  $\big(u_-(t,x;u_0),g_-(t;u_0)\big)$ denotes the unique solution of equation \eqref{eqfunbdn} 
with initial value $u_-(0,x) = u_0(x)$ in $[g_0,\infty)$; for any finite pair $g_0<h_0$ and any $u_0\in\mathcal{H}(g_0,h_0)$, we use $\big(u(t,x;u_0),g(t;u_0),h(t;u_0)\big)$ to denote the unique solution of equation 
\eqref{eqf} with initial value $u(0,x) = u_0(x)$ in $[g_0, h_0]$. 
Finally for any $v_0\in C(\R)\cap L^\infty(\R)$, we let $v(t,x;v_0)$ denote the unique solution of the Cauchy problem \eqref{cauchy} with initial function $v_0(x)$.
 
Let $p(t,x)$ be the unique positive solution for problem~\eqref{psteady11}, and let $\mathcal{C}$ be the subset of $C(\R)$ defined by
\begin{equation*}\left.\baa{ll}
\mathcal{C}:=\Big\{\varphi\in  C(\R):&\!\!\! \hbox{there exists } g_0\in [-\infty,\infty) \hbox{ and } h_0\in (-\infty,\infty] \hbox{ with } g_0 < h_0 \hbox{ such that } \vspace{3pt}\\
&\!\!\! 0<\varphi(x)\leq p(0,x)\hbox{ for }  x\in(g_0,h_0), \hbox{ and }  \varphi(x)=0 \hbox{ for } x\in\R\setminus [g_0,h_0] \Big\}.\eaa\right.
\end{equation*}
For the sake of convenience, for any given $\varphi\in\mathcal{C}$ with $g_0\in [-\infty,\infty)$ and $h_0\in (-\infty,\infty]$ such that $\varphi(x)>0$ if and only if $g_0<x<h_0$, we call $g_0$ the left supporting point of $\varphi$, and $h_0$ the right supporting point of $\varphi$.

We now define an operator $U$ generated by the Poincar\'{e} map of the solution to  problems \eqref{eqf}, \eqref{eqfunbd},  \eqref{eqfunbdn} or \eqref{cauchy}, depending on the nature of  $\varphi\in \mathcal{C}$ in the following way.
 Suppose that $\varphi$ has left supporting point $g_0$ and right supporting point $h_0$. 
\begin{itemize}
\item If $-\infty<g_0< h_0<\infty$, then 
\begin{equation*}
U[\varphi](x):=\left\{\baa{ll}
u(\omega,x;\varphi),& \hbox{ for }\, g(\omega; \varphi)\leq x \leq h(\omega; \varphi),\vspace{3pt}\\
0,&\hbox{ for }\, x> h(\omega; \varphi) \hbox{ or } x< h(\omega; \varphi),\eaa\right.
\end{equation*} 
where $(u,g,h)$ is the unique solution of \eqref{eqf} with initial function $\varphi$;
\item
if $-\infty=g_0< h_0<\infty$, then 
\begin{equation*}
U[\varphi](x):=\left\{\baa{ll}
u_+(\omega,x;\varphi),& \hbox{ for }\, x \leq h_+(\omega; \varphi),\vspace{3pt}\\
0,&\hbox{ for }\, x> h_+(\omega; \varphi),\eaa\right.
\end{equation*} 
where $(u_+,h_+)$ is the unique solution of \eqref{eqfunbd} with initial function $\varphi$;
\item if $-\infty<g_0< h_0=\infty$, then 
\begin{equation*}
U[\varphi](x):=\left\{\baa{ll}
u_-(\omega,x;\varphi),& \hbox{ for }\, x\geq g_-(\omega; \varphi),\vspace{3pt}\\
0,&\hbox{ for }\, x< g_-(\omega; \varphi),\eaa\right.
\end{equation*} 
where $(u_-,g_-)$ is the unique solution of \eqref{eqfunbdn} with initial function $\varphi$;
\item if $g_0=-\infty$ and $h_0=\infty$, then 
\begin{equation*}
U[\varphi](x):=v(\omega,x;\varphi)\, \hbox{ for all  }\, x\in\R,
\end{equation*} 
where $v$ is the unique solution of \eqref{cauchy} with initial function $\varphi$.
\end{itemize}

By the conclusions in Part 1 (\cite{ddl}) and hypothesis (H) here, it is easy to check that $U$ maps $\mathcal{C}$ into itself and has the following properties:

\begin{itemize}
\item[(A1)] $U:\mathcal{C}\to \mathcal{C}$ is order-preserving in the sense that $U[\varphi_1](x)\geq U[\varphi_2](x)$ for $x\in\R$ whenever $\varphi_1(x)\geq \varphi_2(x)$ for $x\in\R$.

\item[(A2)] $U$ is periodic with respect to $L\Z$ in the sense that $T_y\big[U[\varphi]\big]=U\big[T_y[\varphi]\big]$ for all $\varphi\in\mathcal{C}$ and $y\in L\Z$, where $T_y:\mathcal{C}\to\mathcal{C}$ is the translation operator defined by $T_y[\varphi]=\varphi[\cdot-y]$.

\item[(A3)] $U:\mathcal{C}\to \mathcal{C}$ is continuous in the sense that for any sequence $\varphi_n\in \mathcal{C}$ with left supporting points $g_n\in [-\infty,\infty)$ and right supporting points $h_n\in (-\infty,\infty]$, and any $\varphi\in \mathcal{C}$ with left supporting point $g\in [-\infty,\infty)$ and right supporting point $h\in(-\infty,\infty]$, if $\varphi_n(x)$ converges to $\varphi(x)$ locally uniformly for $x\in\R$ as $n\to\infty$, and $g_n$ converges to $g$,  $h_n$ converges to $h$ as $n\to\infty$, then $U[\varphi_n](x)$ converges to $U[\varphi](x)$ locally uniformly in $x\in\R$ as $n\to\infty$.

 \item[(A4)] $U:\mathcal{C}\to \mathcal{C}$ is monostable in the sense that $0$ and $p(0,x)$ are the only fixed points of $U$ in $\mathcal{C}$. Moreover, if $w\in\mathcal{C}$ and $w(x)\geq \epsilon$ for some $\epsilon>0$, then $\lim_{n\to\infty} U^n[w](x)=p(0,x)$ uniformly in $x\in\R$. 

\end{itemize}

Moreover, we have the following
 comparison principle, which follows easily from the above properties and an induction argument.

\begin{pro}\label{opcompare}
Let $U_1$ and $U_2$ be two order-preserving operators defined on $\mathcal{C}$ as described above. Suppose that the sequence $(v_n)_{n\in\N}\subset \mathcal{C}$ satisfies $v_{n+1}(x)\geq U_1[v_n](x)$ for $x\in\R$, and that the sequence $(u_n)_{n\in\N}\subset \mathcal{C}$ satisfies $u_{n+1}(x)\leq U_2[u_n](x)$ for $x\in\R$. Suppose also that $U_1[\varphi](x)\geq U_2[\varphi](x)$ for all $\varphi\in\mathcal{C}$, $x\in\R$ and that $v_0(x)\geq u_0(x)$ for $x\in\R$. Then $v_n(x)\geq u_n(x)$ for all $x\in\R$, $n\in\N$.
\end{pro}

\section{Existence of spreading speed for problems \eqref{eqfunbd} and \eqref{eqfunbdn}}\label{sec3}
In this section, we give a detailed proof for the fact that the free boundary problem \eqref{eqfunbd} admits a spreading speed in the rightward direction. The existence of leftward spreading speed for problem \eqref{eqfunbdn} 
follows from the result for \eqref{eqfunbd} with $f(t,x,u)$ replaced by $f(t,-x,u)$.    Let us recall that  $f$ always satisfies \eqref{zero}, \eqref{hyp2}, \eqref{period} and {\rm (H)}.

First, we define a set $\mathcal{M}$ consisting of functions $\phi(\xi,x)$ in $ C(\R^2)$ with the following properties:
\begin{equation}\label{M}
\left.\begin{array}{l}
\mbox{(a) \; For each $\xi\in\R$, $\phi(\xi,x)$ is nonnegative and $L$-periodic in $x$;} \\
\mbox{(b) \;  $\phi(\xi,x)$ is uniformly continuous in $(\xi,x)\in\R^2$;}\\
\mbox{(c) \;  For each fixed $x$, $\phi(\xi,x)$ is nonincreasing in $\xi$;  }
\\
\mbox{(d) \; For any $\xi\in\R$, there exists a real number $H_0=H_0(\xi)$ such that}\\
\hspace{1cm}\mbox{ $\phi(\xi+x,x)>0$ for $x<H_0$ and $\phi(\xi+x,x)=0$ for $x\geq H_0$;}
\\
\mbox{(e) \; $0<\phi(-\infty,x)<p(0,x)$ for all $x\in\R$.
}
\end{array}
\right\}
\end{equation}

For each $\alpha\in\R$ and $\beta\in (0,1)$, let
\[
\tau(\xi):=\beta\frac{\max\{\alpha-\xi,0\}}{\max\{\alpha-\xi, 0\}+1}.
\]
Clearly the function $\tau(\xi)p(0,x)$ belongs to $\mathcal{M}$.

\smallskip

The above properties of $\phi\in\mathcal{M}$  imply the following one.

\begin{lem}\label{maxmin}
For any $\phi\in\mathcal{M}$, there exists $H_2\geq H_1$ such that $\phi(\xi,\cdot)\equiv 0$ if $\xi\geq H_2$ and $\phi(\xi,x)>0$ for all $x\in\R$ if $\xi<H_1$. 
\end{lem}

\begin{proof}
 By property (d) in \eqref{M}, for any $\xi_0\in\R$, there exists $H_0=H_0(\xi_0)$ such that $\phi(\xi_0+x,x)>0$ if and only if $x< H_0$. By (c) we have
\[
\phi(\xi_0+H_0-2L,x)\geq \phi(\xi_0+x,x)>0 \mbox{ for } x\in[H_0-2L, H_0-L].
\]
By (a) this implies that $\phi(\xi_0+H_0-2L, x)>0$ for all $x\in\R$. Hence, in view of (c),
\[
\phi(\xi,x)>0 \mbox{ for } \xi< H_1:=\xi_0+H_0-2L,\; x\in\R.
\]

Similarly, 
\[
\phi(\xi_0+H_0+L,x)\leq \phi(\xi_0+x,x)=0 \mbox{ for } x\in[H_0, H_0+L].
\]
By (a) this implies that $\phi(\xi_0+H_0+L, x)=0$ for all $x\in\R$. Hence, in view of (c),
\[
\phi(\xi,x)=0 \mbox{ for } \xi\geq  H_2:=\xi_0+H_0+L,\; x\in\R.
\]
The proof of Lemma~\ref{maxmin} is thereby complete.
\end{proof}

For any $\phi\in\mathcal{M}$, and any fixed $\xi_0\in\R$, by \eqref{M}, the function $x\to \phi(\xi_0+x, x)$ belongs to $\mathcal{C}$, with right supporting point  $H_0=H_0(\xi_0)\in\R$. Therefore $U[\phi(\xi_0+\cdot,\cdot)](x)$ is well-defined, and
\[
U[\phi(\xi_0+\cdot,\cdot)](x):=\left\{\baa{ll}
u_+(\omega,x; \phi(\xi_0+\cdot,\cdot)),& \hbox{ if }\, x\leq h_+(\omega; \phi(\xi_0+\cdot,\cdot)),\vspace{3pt}\\
0,&\hbox{ if }\, x> h_+(\omega; \phi(\xi_0+\cdot,\cdot)).\eaa\right.
\]

For $(\xi, y)\in\R^2$,
we now define the operator $Q_+$ on $\mathcal{M}$ by 
\begin{equation}\label{defirq}
Q_+[\phi](\xi,y):=U[\phi(\xi-y+\cdot,\cdot)](y)\,\,\hbox{ for any }\, \phi\in \mathcal{M}. 
\end{equation}
As will become clear below, we will make use of $Q_+$ and its iterations to determine the rightward spreading speed.

We now examine the properties of $Q_+$. 
\begin{lem}\label{wellde}
$Q_+$ maps $\mathcal{M}$ to $\mathcal{M}$, and $Q_+$ is order preserving in the sense that $Q_+[\phi_1]\geq Q_+[\phi_2]$ whenever $\phi_1\geq \phi_2$ in $\mathcal{M}$.
\end{lem}   

\begin{proof}
To prove that $Q_+$ maps $\mathcal{M}$ into itself, it is sufficient to prove that, for any $\phi\in\mathcal{M}$, $Q_+[\phi](\xi,y)$  has the properties stated in \eqref{M}. In what follows, we divide the proof into five steps, and in each step, we show one property.   

{\it Step 1: We prove that $Q_+[\phi](\xi,y)$ is  nonnegative and is $L$-periodic in $y$.} The non-negativity is clear from the definition. It remains to show that it is $L$-periodic in $y$.

Since the operator $U$ is periodic with respect to $L\Z$ in the sense of (A2), it is easy to check that 
$$U[\phi(\cdot+\xi,\cdot)](y+L)= U[\phi(\cdot+L+\xi,\cdot+L)](y)\,\,\hbox{ for all }\, \xi\in\R,\,y\in\R.$$
This together with the $L$-periodicity of $\phi$ in the second variable implies that 
\begin{equation*}
\begin{split}
Q_+[\phi](\xi,y+L)&=U[\phi(\cdot-y-L+\xi,\cdot)](y+L)\\
&= U[\phi(\cdot-y+\xi,\cdot+L)](y)\\
&=U[\phi(\cdot-y+\xi,\cdot)](y)\\
&=Q_+[\phi](\xi,y)\,\,\hbox{ for all } \xi\in\R,\, y\in\R.
\end{split}
\end{equation*}
Thus, $Q_+[\phi](\xi,y)$ is $L$-periodic in $y$.

{\it Step 2: We prove that $Q_+[\phi](\xi,y)$ is uniformly continuous in $(\xi,y)\in\R^2$}.  This  is a consequence of
 the continuity of the operator $U$ in the sense of (A3) and the uniform continuity of the function $\phi(\xi,x)$ with respect to $(\xi,x)\in\R^2$.

{\it Step 3: We prove that $Q_+[\phi](\xi,y)$ is nonincreasing in $\xi$.}  By (A1) we have 
$$Q_+[\phi_1](\xi,y)\geq Q_+[\phi_2](\xi,y)\,\hbox{ for all } (\xi,y)\in\R^2,$$ 
whenever $\phi_1\geq \phi_2$ in $\mathcal{M}$. This together with the property that $\phi(\xi,x)$ is nonincreasing in $\xi$ implies that  $Q_+[\phi](\xi,y)$ is also nonincreasing in $\xi$.

{\it Step 4: We prove that for any $\xi\in\R$, there exists $H\in\R$ depending on $\xi$ such that 
 $Q_+[\phi](\xi+y,y)=0$ if and only if $x\geq H$.}  
As a matter of fact, for any given $\xi\in\R$, let $\big(u_+(t,x),h_+(t)\big)$ be the solution of equation \eqref{eqfunbd} with initial value 
$$u_+(0,x) = \phi(x+\xi,x)\,\hbox{ in } \big(-\infty, H_0(\xi)\big].$$ 
Set $H=h_+(\omega)$. Then by the definitions of $Q_+$ and $U$, it is easy to see that $H$ is the desired critical number.

{\it Step 5: We prove that the limit $Q_+[\phi](-\infty,y)$ exists and $0<Q_+[\phi](-\infty,y)<p(0,y)$ for all $y\in\R$.}  To do so, we choose a sequence $(\xi_n)_{n\in\N}\subset\R$ such that  $\xi_n$ is nonincreasing in $n$ and that $\xi_n\to-\infty$ as $n\to\infty$. Due to the property (b) in \eqref{M}, $\phi(-\infty,x)$ is continuous in $x$.  Furthermore, since $\phi(\xi_n+x,x)$ is nondecreasing in $n$, and equi-continuous in $x$,  $\phi(\xi_n+x,x)$ converges to $\phi(-\infty,x)$ locally uniformly in $x\in\R$ as $n\to\infty$. It then follows from the continuity of the operator $U$ in the sense of (A3) that 
$$U[\phi(\xi_n+\cdot,\cdot)](y)\to U[\phi(-\infty,\cdot)](y) \,\hbox{ as } n\to\infty \hbox{ locally uniformly in } y\in\R,  $$ 
that is, 
$$Q_+[\phi](\xi_n+y,y) \to U[\phi(-\infty,\cdot)](y)\,\hbox{ as }\,n\to\infty\,\hbox{ locally uniformly in }\,y\in\R.$$
 Since the nonincreasing sequence $(\xi_n)_{n\in\N}$ can be chosen arbitrarily,  the limit $Q_+[\phi](-\infty,y)$ exists and 
 $$Q_+[\phi](-\infty,y)=U[\phi(-\infty,\cdot)](y).$$ Furthermore, 
by the property (e) in \eqref{M}, it follows from the parabolic strong maximum principle that 
$$0<U[\phi(-\infty,\cdot)](y)<p(0,y)\, \hbox{ for all } y\in\R.$$ Thus, $Q_+[\phi](\xi,y)$ possesses the property stated in (e). 

Therefore, the operator $ Q_+$ maps $\mathcal{M}$ into itself. Lastly, the order-preserving property of $Q_+$ follows easily from that of $U$ stated in (A1). The proof of Lemma~\ref{wellde} is thereby complete.
\end{proof}


We now fix an arbitrary $\phi\in\mathcal{M}$ and, for any $c\in\R$, we define the sequence $\{(a_n^c,H_n^c)\}_{n\in\N}$ by the following recursions
\begin{equation}\label{recuran}
a_{n}^c(\xi,x)=\max\Big\{\phi(\xi,x), \,Q_+[a_{n-1}^c](\xi+c,x) \Big\} 
\end{equation}
and 
\begin{equation}\label{recurhn}
H_{n}^c(\xi)=\max\Big\{H_0(\xi),\,h_+\big(\omega;a_{n-1}^c(\cdot+\xi+c,\cdot)\big) \Big\}
\end{equation}
where $  a_{0}^c(\xi,x)=\phi(\xi,x)$
and $H_0(\xi)$ is the real number such that $\phi(\xi+x,x)=0$ if and only if $x\geq H_0(\xi)$. 

By Lemma~\ref{wellde}, it is easily seen that for each $n\in\N$, $a_{n}^c\in\mathcal{M}$.  Further properties are given below.

\begin{lem}\label{propan} The following statements are valid:

\begin{itemize} 
\item[(i)] For any fixed $n\in\N$, $0\leq a_{n}^c(\xi,x) \leq  p(0,x)$ for all $\xi\in\R$, $x\in\R$, and $a_{n}^c(\xi+x,x)=0$ if and only if $x\geq H_{n}^c(\xi)$.

\item[(ii)] The sequence $a_n^c(\xi,x)$ is nondecreasing in $n$, nonincreasing in $\xi$ and $c$, $L$-periodic in $x$, and the sequence $H_n^c(\xi)$ is nondecreasing in $n$.

\item[(iii)] For any fixed $n\in\N$,  $a_n^c(\xi,x)$ is uniformly continuous in $(c,\xi,x)\in\R^3$ and $H_n^c(\xi)$ is uniformly continuous in $(c,\xi)\in\R^2$.

\item[(iv)] $\big\{ a_n^c(\xi+x,x):\,\xi\in\R,\,c\in\R,\,n\in\N \big\}$ is a family of equicontinuous functions of $x\in\R$.
\end{itemize}
\end{lem}

\begin{proof}
We only give the proof for the statement (iv), since the statements (i)-(iii) are easily proved by an induction argument.
For any fixed $\xi\in\R$ and $c\in\R$, let $\big(u_+(t,x),h_+(t)\big)$ be the solution of equation \eqref{eqfunbd} with initial value 
$$u_+(0,x) = \phi(x+\xi+c,x)\,\hbox{ in } \big(-\infty, H_0(\xi+c)\big].$$ It follows from \cite[Remark~2.12]{ddl}, that the function $u_+(\omega,x)$ is Lipschitz continuous in $(-\infty,h_+(\omega)]$, and the Lipschitz constant depends only on $\omega$ and $\|\phi\|_{L^{\infty}(\R^2)}$ (and hence, it is independent of $c$ and $\xi$). By the definition of $U$, we know $Q_+[\phi](x+\xi+c,x)$ is Lipschitz continuous in $x\in\R$ with the same Lipschitz constant. In a similar way, one concludes that for any $n\in\N$,  $Q_+[a_n^c](x+\xi+c,x)$ is Lipschitz continuous in $x\in\R$ with Lipschitz constant depending only on $\omega$ and $\|a_n^c\|_{L^{\infty}(\R^2)}$. Since the sequence $\{a_n^c\}_{n\in\N}$ is uniformly bounded, it follows that the family
$$\Big\{ Q_+[a_n^c](x+\xi+c,x):\,\xi\in\R,\,c\in\R,\,n\in\N \Big\}$$ is uniformly Lipschitz continuous in $x\in\R$. Moreover, since $\phi(\xi,x)$ is uniformly continuous in $(\xi,x)\in\R^2$, we have $\phi(\xi+x,x)$ is uniformly continuous in $x\in\R$ uniformly in $\xi\in\R$. This implies the equicontinuity of the family 
$\big\{ a_n^c(\xi+x,x):\,\xi\in\R,\,c\in\R,\,n\in\N \big\}.$ 
The proof of Lemma~\ref{propan} is thereby complete.
\end{proof}

Next, for any fixed $n\in\N$ and $c\in\R$, we consider the limits of $a_n^c(\xi,x)$ as $\xi\to\pm\infty$. By Lemma \ref{maxmin}, for each fixed $n$ and $c$, $a_n^c(\xi,\cdot)\equiv 0$ for all large $\xi$. Hence $a_n^c(+\infty,x)=0$. The following lemma is concerned with 
$$\alpha_n(x)=\alpha_n^c(x):=a_n^c(-\infty,x), \, x\in\R.$$

\begin{lem}\label{infcau}
For each $n\in\N$, $\alpha_n(x)$ is $L$-periodic in $x$, nondecreasing in $n$, and $\alpha_n(x)$ converges to $p(0,x)$ as $n\to\infty$ uniformly in $x\in\R$. 
\end{lem}

\begin{proof}
It is clear that $\alpha_n(x)$ is $L$-periodic in $x$ and nondecreasing in $n$, since $a_n^c(\xi,x)$ has these properties.
Next, we prove the convergence of $\alpha_n(x)$ as $n\to\infty$. Since the operator $Q_+$ is order preserving and is also translation invariant with respect to the variable $\xi$, one concludes by an induction argument that, for any $\phi\in\mathcal{M}$, 
\begin{equation}\label{auxicom}
Q_+^n[\phi](\xi+nc,x)\leq a^c_{n}(\xi,x) \leq p(0,x)\quad \hbox{for all }\,\xi\in\R,\,x\in\R,\,n\in\N.
\end{equation}
Now, for any fixed $n\in\N$, passing to the limit $\xi\to-\infty$ in the above inequality, we obtain  
$$Q_+^n[\phi](-\infty,x)\leq \alpha_n(x)\leq p(0,x)\,\hbox{ in } \R.$$ 
Furthermore, by the analysis in Step 5 of the proof of Lemma~\ref{wellde}, there holds
$$Q_+^n[\phi](-\infty,x)=U^n[\phi(-\infty,\cdot)](x) \, \hbox{ for each } n\in\N.$$ 
Since $\phi(-\infty,\cdot)$ is positive, $L$-periodic, it follows from the property (A4) that 
$$\lim_{n\to\infty}U^n[\phi(-\infty,\cdot)](x)=p(0,x) \, \hbox{ uniformly in } x\in\R,$$ whence $\alpha_n(x)$ converges to $p(0,x)$ as $n\to\infty$ uniformly in $x\in\R$. The proof of Lemma~\ref{infcau} is thereby complete.
\end{proof}

We now consider the limit function of $a_n^c(\xi,x)$ as $n\to \infty$.

\begin{lem}\label{infian}
The following statements are valid:
\begin{itemize}
\item[(i)] 
For each fixed $c\in\R$ and $\xi\in\R$, there exist a function $a^c(\xi+\cdot,\cdot)\in C(\R)$ and some $H^c(\xi)\in [H_0(\xi),+\infty]$ such that 
$\lim_{n\to\infty}H^c_n(\xi)= H^c(\xi)$ and 
\[\mbox{$\lim_{n\to\infty}a^c_n(\xi+x,x)=a^c(\xi+x,x)$ locally uniformly in $x\in\R$.}
\]

\item[(ii)]
For each fixed $c\in\R$ and $\xi\in\R$, if $H^c(\xi)<\infty$, then $a^c(\xi+x,x)=0$ if and only if $x\geq H^c(\xi)$, and if $H^c(\xi)=\infty$, then $a^c(\xi+x,x)>0$ for all $x\in\R$. Moreover, 
\begin{equation}\label{hxiinf}
H^c(\xi+kL)=H^c(\xi)-kL\, \hbox{ for all }\,k\in\Z.
\end{equation} 

\item[(iii)] The function $a^c(\xi,x)$ is nonincreasing in $c\in\R$ and $\xi\in\R$, $L$-periodic in $x\in\R$, and $a^c(-\infty,x)\equiv p(0,x)$. Moreover, for any fixed $c\in\R$ and $\xi\in\R$, 
\begin{equation}\label{recurlima}
a^c(\xi+x,x)=\max\Big\{\phi(\xi+x,x), \, U[a^c(\cdot+\xi+c,\cdot)](x) \Big\}\,\hbox{ for all } x\in\R. 
\end{equation}
\end{itemize}
\end{lem}

\begin{proof}
(i) For any fixed $c\in\R$ and $\xi\in\R$, due to the monotonicity properties stated in Lemma~\ref{propan} (ii),  we may define 
$$H^c(\xi):=\lim_{n\to\infty}H^c_n(\xi) \in [H_0(\xi),+\infty],$$ 
$$a^c(\xi,x):=\lim_{n\to\infty}a^c_n(\xi,x),\;\; \,(\xi, x)\in\R^2.$$
 Furthermore, by Lemma~\ref{propan} (iv) and the Arzel\`{a}-Ascoli Theorem, 
 $$a^c_n(\xi+x,x)\to a^c(\xi+\cdot,\cdot) \hbox{ as } n\to\infty  \hbox{ locally uniformly in } x\in\R.$$ 
 This in particular implies that $a^c(\xi+\cdot,\cdot)\in C(\R)$. We have thus proved all the conclusions in (i). 

(ii) Fix $c\in\R$ and $\xi\in\R$. We first consider the case where $H^c(\xi)<\infty$. For any given $x<H^c(\xi)$, since $H_n^c(\xi)$ converges to $H^c(\xi)$ as $n\to\infty$, there exists some $n_0$ such that $x<H_{n}^c(\xi)$ for all $n\geq n_0$. Then by Lemma~\ref{propan} (i), we have $a_n^c(\xi+x,x)>0$ for all $n\geq n_0$. This together with the fact that $a_n^c(\xi+x,x)$ is nondecreasing in $n$ implies that $a^c(\xi+x,x)>0$. 

On the other hand, for any given $x\geq H^c(\xi)$, since $H_n^c(\xi)$ is nondecreasing in $n$, there holds $x\geq H_n^c(\xi)$ for all $n\in\N$, whence $a_n^c(\xi+x,x)=0$ by Lemma~\ref{propan} (i) again. Therefore,  $a^c(\xi+x,x)=0$. Similarly, one concludes that if $H^c(\xi)=\infty$, then $a^c(\xi+x,x)>0$ for all $x\in\R$. 

We now show the equality \eqref{hxiinf}. It suffices to prove that  
$$ H^c_n(\xi+kL)=H^c_n(\xi)-kL\, \hbox{ for all }\,k\in\Z,\,n\in\N.$$
By the definition of $H^c_n(\xi)$,  
$a_n^c(\xi+x,x)= 0$ if and only if $x\geq H^c_n(\xi)$. Since $a_n^c$ is $L$-periodic in its second variable, it follows that 
$$a_n^c(\xi+kL+x,x)=a_n^c(\xi+kL+x,kL+x),$$
and hence,  $a_n^c(\xi+kL+x,x)=0$ if and only if $x+kL\geq H^c_n(\xi)$. Thus,  $H^c_n(\xi+kL)=H^c_n(\xi)-kL$ and \eqref{hxiinf} is proved.

(iii) Since for each fixed $n\in\N$, $a^c_n(\xi,x)$ is nonincreasing in $\xi\in\R$ and $c\in\R$, $L$-periodic in $x$, its limit $a^c(\xi,x)$  also possesses these properties.
This in particular implies that the limits $a^c(\pm \infty,x)$ exist. Furthermore, letting $n\to\infty$ followed by sending $\xi\to-\infty$ in the first inequality of Lemma~\ref{propan} (i), we obtain 
$$0\leq a^c(-\infty,x)\leq p(0,x)\,\hbox{ for } x\in\R.$$ 
On the other hand, since 
$$a^c_n(-\infty,x)\leq a^c(-\infty,x) \,\hbox{ for all }n\in\N, \,x\in\R,$$ 
and since $a^c_n(-\infty,x)$ converges to $p(0,x)$ uniformly in $x\in\R$ as $n\to\infty$ by Lemma~\ref{infcau}, it follows that $a^c(-\infty,x)\equiv p(0,x)$. 
Finally, for any given $c\in\R$ and $\xi\in\R$,  by \eqref{defirq} and \eqref{recuran}, we have
\begin{equation*}
a_{n+1}^c(\xi+x,x)=\max\Big\{\phi(\xi+x,x), \, U[a_n^c(\cdot+\xi+c,\cdot)](x) \Big\}\,\hbox{ for all } x\in\R,\,n\in\N. 
\end{equation*}
Since $a^c_n(\xi+x,x)$ converges to $a^c(\xi+x,x)$ locally uniformly in $x\in\R$ and $H^c_n(\xi)$ converges to $H^c(\xi)$ as $n\to\infty$, it follows from the property (A3) that  
$$U[a_n^c(\cdot+\xi+c,\cdot)](x)\to U[a^c(\cdot+\xi+c,\cdot)](x)\,\hbox{ locally uniformly in }\,x\in\R\,\hbox{ as }\,n\to\infty.$$ 
Then taking the limit $n\to\infty$ in the above equality, we arrive at \eqref{recurlima}. 
The proof of Lemma~\ref{infian} is thereby complete.
\end{proof}

By Lemma~\ref{propan} (iv) and Lemma~\ref{infian} (iii), it is easily seen that the limit $a^c(\infty,x)$ exists and it is continuous in $x\in\R$. The following two lemmas supply some crucial  properties of $a^c(\infty,x)$, which are the key to obtain the spreading speed.

\begin{lem}\label{acinfty}
Either $a^c(\infty,x)\equiv 0$ or $a^c(\infty,x)\equiv p(0,x)$.
\end{lem}
\begin{proof}
Fix $c\in\R$. If there exists $\xi_0\in\R$ such that $H^c(\xi_0)<\infty$, then it follows from Lemma~\ref{infian} (ii) that $a^c(x+\xi_0,x)=0$ if $x\geq H^c(\xi_0)$. This implies that, for any $x_0\in\R$, there exists $k_0\in\N$ large enough such that 
$$a^c(\xi_0+x_0+kL,x_0+kL)=0 \,\hbox{ for all } k\geq k_0.$$ 
Since $a^c$ is $L$-periodic in its second variable, it follows that 
$$a^c(\xi_0+x_0+kL,x_0)=0 \,\hbox{ for all } k\geq k_0.$$ 
Sending $k\to\infty$ yields $a^c(\infty,x_0)=0$. Since $x_0$ is arbitrary, it follows that $a^c(\infty,\cdot)\equiv 0$. 

Otherwise  $H^c(\xi)=\infty$ for every $\xi\in\R$.  Letting $\xi\to+\infty$ in \eqref{recurlima}, by the continuity property (A3), we obtain 
\[
a^c(\infty,x)=\max\big\{\phi(+\infty, x), U[a^c(\infty,\cdot) ](x)\big\}= U[a^c(\infty,\cdot) ](x)
\mbox{
for $x\in\R$,}
\]
 since $\phi(+\infty, x)=0$. Thus $a^c(\infty,x)$ is an equilibrium of the operator $U$ in $\mathcal{C}$. Due to the property (A4), it follows that either $a^c(\infty,x)\equiv p(0,x)$ or  $a^c(\infty,x)\equiv 0$. 

Consequently, either $a^c(\infty,x)\equiv p(0,x)$ and $H^c(\xi)\equiv\infty$, or  $a^c(\infty,x)\equiv 0$. 
\end{proof}

\begin{lem}\label{comkey}
Let $H_2$ be a real number such that $\phi(\xi,\cdot)\equiv 0$ for $\xi\geq H_2$ $($whose existence is given in Lemma~\ref{maxmin}$)$. Then
$a^c(\infty,x)\equiv p(0,x)$ if and only if there is some $n_0\in\N$ such that 
\begin{equation*}
a^c_{n_0}(H_2,x)>\phi(-\infty,x)\,\,\hbox{ for all }\, x\in\R.
\end{equation*}

\end{lem}

\begin{proof}
If $a^c(\infty,x)=p(0,x)$, then by the monotonicity of $a^c(\xi,x)$ in $\xi$ and $a^c(\xi,x)\leq p(0,x)$, we have $\phi(\xi,x)\equiv p(0,x)$ and hence 
$$a^c(\xi+x,x)=p(0,x)\, \hbox{ for all } \xi\in\R,\, x\in\R.$$ 
By choosing $\xi=H_2$, it follows from Lemma~\ref{infian} (i) that $a_n^c(H_2+x,x)$ converges to $p(0,x)$ locally uniformly in $x\in\R$ as $n\to\infty$. This together with the assumption that $p(0,x)>\phi(-\infty, x)$ for all $x\in\R$, implies that there is some $n_0\in\R$ such that 
$$ a^c_{n_0}(H_2+x,x)>\phi(-\infty,x)\,\,\hbox{ for all }\, x\in [0,L]. $$
Since $a^c_{n_0}(\xi,x)$ is nonincreasing in $\xi$, we obtain $a^c_{n_0}(H_2,x)>\phi(-\infty,x)$ for all $x\in [0,L]$. Furthermore, since both $a^c_{n_0}(H_2,x)$ and  $\phi(-\infty,x)$ are $L$-periodic in $x$, it follows that  $a^c_{n_0}(H_2,x)>\phi(-\infty,x)$ for all $x\in\R$.

Conversely, suppose that there is some $n_0\in\N$ such that $a^c_{n_0}(H_2,x)>\phi(-\infty,x)$ for all $x\in\R$. Since both $a_{n_0}^c(\xi,x)$  and $\phi(-\infty,x)$ are $L$-periodic in $x$, and $a_{n_0}^c(\xi,x)$ is uniformly continuous in $(\xi,x)\in\R^2$,  there exists a positive constant $\delta>0$ such that 
$$a_{n_0}^c(H_2+\delta,x)>\phi(-\infty,x)\, \hbox{ for all } x\in\R.$$ 
Furthermore, since $a_{n_0}^c(\xi,x)$ and $\phi(\xi,x)$ are nonincreasing in $\xi$ and since $\phi(\xi,x)=0$ for all $\xi\geq H_2$, it follows that
\[
\mbox{ $a_{n_0}^c(\delta+\xi,x)\geq \phi(\xi,x)$ for all $(\xi,x)\in\R^2$.}
\]

Now we claim that 
\begin{equation}\label{induac}
a_{n_0+k}^c(\delta+\xi,x)\geq a_k^c(\xi,x)\quad \hbox{for all }\, (\xi, x)\in\R^2,\,k\in\N.
\end{equation}
We already know that the above inequality holds for $k=0$. Suppose that \eqref{induac} holds for some integer $k=k_0\geq 0$. Then 
\begin{equation*}
\begin{split}
a_{n_0+k_0+1}^c(\xi+\delta,x)&=\max\Big\{\phi(\xi+\delta,x), \,Q_+[a_{n_0+k_0}^c](\xi+\delta+c,x)\Big\}\\
&\geq Q_+[a_{n_0+k_0}^c](\xi+\delta+c,x)\\
& \geq Q_+[a_{n_0}^c](\xi+c,x) \;\mbox{ for all $(\xi,x)\in\R^2$.}
\end{split}
\end{equation*}
We also have
\[
a_{n_0+k_0+1}^c(\xi+\delta,x)\geq a_{n_0}^c(\xi+\delta,x)\geq \phi(\xi,x) \mbox{ for } (\xi,x)\in\R^2.
\]
Therefore, for $(\xi,x)\in\R^2$,
\[
a_{n_0+k_0+1}^c(\xi+\delta,x)\geq\max\Big\{\phi(\xi,x),\; Q_+[a_{n_0}^c](\xi+c,x)\Big\}=a_{n_0+1}^c(\xi,x).
\]
This implies that \eqref{induac} holds for all $k\in\N$. 

Passing to the limit  $k\to\infty$ in \eqref{induac} gives 
$$a^c(\xi+\delta,x)\geq a^c(\xi,x)\, \hbox{ for all } (\xi,x)\in\R^2.$$ 
This together with the fact that $a^c(\xi,x)$ is nonincreasing in $\xi$ implies that $a^c(\xi,x)$ is independent of $\xi$. Furthermore, since $a^c(-\infty,x)\equiv p(0,x)$ by Lemma~\ref{infian} (iii), it follows that $a^c(\infty,x)\equiv p(0,x)$. The proof of Lemma~\ref{comkey} is thereby complete. 
\end{proof}

Define
\begin{equation}\label{defic}
c_+=\sup\Big\{c\in\R:\, a^c(\infty,x)\equiv p(0,x)\Big\}.
\end{equation}
If there does not exist $c\in\R$ such that $a^c(\infty,x)\equiv p(0,x)$, then we define $c_+=-\infty$.

\begin{lem}\label{proc}
$c_+>-\infty$, and $a^c(\infty,x)\equiv p(0,x)$ if $c< c_+$, $a^c(\infty,x)\equiv 0$ if $c\geq  c_+$.
\end{lem}
\begin{proof}
We first prove that $c_+>-\infty$. It follows from the property (A4) that $U^n[\phi(-\infty,\cdot)](x)$ converges to $p(0,x)$ as $n\to\infty$ uniformly in $x\in\R$, and hence, for any small positive constant $\epsilon$, there exists some $n_0\in\N$ large enough such that 
$$U^{n_0}[\phi(-\infty,\cdot)](x)-\epsilon>\phi(-\infty,x)\,\hbox{ for all }\,x\in\R.$$
On the other hand, by \eqref{auxicom}, we have
$$a^c_{n_0}(H_2,x) \geq Q_+^{n_0}[\phi](H_2+n_0c,x)  \, \hbox{ for all }\,c\in\R,\,x\in\R, $$
with $H_2$ given in Lemma~\ref{maxmin} such that $\phi(\xi,\cdot)\equiv 0$ 
for $\xi\geq H_2$. Furthermore, since 
$Q_+^{n_0}[\phi](-\infty,x)\equiv U^{n_0}[\phi(-\infty,\cdot)](x)$ (whose proof is the same as that in Step 5 of the proof of Lemma~\ref{wellde}), there exists $c_0<0$ large negative such that  
$$Q_+^{n_0}[\phi](H_2+n_0c_0,x) \geq U^{n_0}[\phi(-\infty,\cdot)](x) -\epsilon \,\hbox{ for all }\,x\in\R. $$
Here we have implicitly used the fact that these two functions of $x$ are $L$-periodic.
Combining the above inequalities, we obtain  
$$a^{c_0}_{n_0}(H_2,x) > \phi(-\infty,x)\, \hbox{ for all }x\in\R. $$
It then follows from Lemma~\ref{comkey} that $a^{c_0}(\infty,x)\equiv p(0,x)$, and hence, $c_+\geq c_0>-\infty$.

Next, since $a^c(\infty,x)$ is nonincreasing in $c$, it follows readily from Lemma~\ref{acinfty} that $a^c(\infty,x)\equiv p(0,x)$ if $c<c_+$. By the definition of $c_+$, we clearly have $a^c(\infty,x)\equiv 0$ for $c>c_+$ when $c_+<+\infty$. 
(We will prove later that $c_+<+\infty$ always holds; see Proposition \ref{r-speed-finite} below.)

To complete the proof, it remains to show that $a^{c_+}(\infty,x)\equiv 0$ when $c_+<+\infty$. Assume by contradiction that $a^{c_+}(\infty,x)\equiv p(0,x)$. Then by Lemma~\ref{comkey} again, there exists some $n_1\in\N$ such that 
$$a^{c_+}_{n_1}(H_2,x)>\phi(-\infty,x)\,\hbox{ for all } \,x\in\R.$$ 
By the continuity of the function 
$a^{c}_{n_1}(H_2,x)$ with respect to $c$, and the fact that it is periodic in $x$, it follows that $a^{c}_{n_1}(H_2,x)>\phi(-\infty,x)$ for $c$ in a neighbourhood of $c_+$, whence $a^{c}(\infty,x)=p(0,x)$, which is in contradiction to the definition of $c_+$. The proof of Lemma~\ref{proc} is now complete.
\end{proof}

The following lemma shows that $c_+$ is independent of the choice of the function $\phi$.

\begin{lem}\label{independc}
Let $\big\{\big(\tilde{a}_n^c(\xi,x), \tilde{H}_n^c(\xi)\big)\big\}_{n\in\N}$ be the sequence obtained from the recursions \eqref{recuran} and \eqref{recurhn} when $\phi$ is replaced by another function $\tilde{\phi}\in\mathcal{M}$ with $\tilde{H}_0:=\tilde{H}_0(\xi)$ such that $\tilde{\phi}(\xi+x,x)=0$ if and only if $x\geq \tilde{H}_0$. Then the limit $\big\{\big(\tilde{a}^c(\xi,x), \tilde{H}^c(\xi)\big)\big\}$ of $\big\{\big(\tilde{a}_n^c(\xi,x), \tilde{H}_n^c(\xi)\big)\big\}$ as $n\to\infty$ satisfies $\tilde{a}^c(\infty,x)=a^c(\infty,x)$ and, for every $\xi\in\R$, either $\tilde{H}^c(\xi)$ and $H^c(\xi)$ are both finite or they are both infinite.
\end{lem}

\begin{proof}
By the same proof as that in Lemma~\ref{infcau}, we see that $\tilde{a}_n^c(-\infty,x)$ converges to $p(0,x)$ as $n\to\infty$ uniformly in $x\in\R$. Since $\phi(-\infty,x)<p(0,x)$ and these functions are $L$-periodic in $x$, there exists $n_0\in\N$ such that $\tilde{a}_{n_0}^c(-\infty,x)>\phi(-\infty,x)$ for all $x\in\R$. This implies that  
$$\tilde{a}_{n_0}^c(\hat{H},x)> \phi(-\infty,x) \,\hbox{ for } x\in\R\, \hbox{ and all large negative } \hat H.$$ 
We fix such a $\hat{H}$ with the additional property that $\hat{H}-H_2\in L\Z$, where $H_2$ is given in Lemma~\ref{maxmin} such that $\phi(\xi,\cdot)\equiv 0$ if $\xi\geq H_2$. 
Since $\tilde{a}_{n_0}^c(\xi,x)$ and $\phi(\xi,x)$ are both nonincreasing in $\xi\in\R$, it follows that 
$$\tilde{a}_{n_0}^c(\xi+\hat{H}-H_2,x)\geq \phi(\xi,x)\, \hbox{ for all }\, \xi\in\R,\,x\in\R.$$ 
Then by an induction argument similar to that used in the proof of Lemma~\ref{comkey}, we obtain 
$$ \tilde{a}_{n_0+k}^c(\xi+\hat{H}-H_2,x)> a_{k}^c(\xi,x)\, \hbox{ for all } \,\xi\in\R,\,x\in\R,\,k\in\N.$$
Taking the limit as $k\to\infty$ in the above inequality yields that 
$$\tilde{a}^c(\xi+\hat{H}-H_2,x)\geq a^c(\xi,x)\,  \hbox{ for all } (\xi,x)\in \R^2,$$ whence 
$$\tilde{a}^c(\infty,x)\geq a^c(\infty,x)\, \hbox{ for all }\,x\in\R, \quad \hbox{and}\quad \tilde{H}^c(\xi)\geq H^c(\xi-\hat{H}+H_2)\,\hbox{ for all }\,\xi\in\R. $$
Furthermore, since $\hat{H}-H_2\in L\Z$, it follows from \eqref{hxiinf} that 
$$\tilde{H}^c(\xi)\geq H^c(\xi)+\hat{H}-H_2\, \hbox{ for all } \xi\in\R.$$

In a similar way, by reversing the roles of $\big(\tilde{a}_n^c(\xi,x), \tilde{H}_n^c(\xi)\big)$ and $\big(a_n^c(\xi,x), H_n^c(\xi)\big)$ in the above arguments, we obtain two real numbers $\bar{H}$ and $\tilde{H}_2$ ($\tilde{H}_2$ is a constant such that $\tilde{\phi}(\xi,\cdot)\equiv 0$ if $\xi\geq \tilde{H}_2$) with $\bar{H}-\tilde{H}_2\in L\Z$ such that 
$$a^c(\xi+\bar{H}-\tilde{H}_2,x)\geq \tilde{a}^c(\xi,x)\,  \hbox{ for all } (\xi,x)\in \R^2,$$ 
and hence 
$$a^c(\infty,x)\geq \tilde{a}^c(\infty,x)\, \hbox{ for all }\,x\in\R, \quad \hbox{and}\quad H^c(\xi)\geq \tilde{H}^c(\xi)
+\bar{H}-\tilde{H}_2\,\hbox{ for all }\,\xi\in\R.$$
The proof of Lemma~\ref{independc} is now complete.
\end{proof}

Summarising the above results we immediately obtain
\begin{pro}\label{disspeed}
Let $c_+$ be given in \eqref{defic}. Then $c_+>-\infty$ and is independent of the choice of $\phi\in\mathcal{M}$
in the recursion \eqref{recuran} leading to $a^c(\infty,x)$.
\end{pro}

Notice that if $\xi$ is replaced by $x+\xi-(n+1)c$ in \eqref{recuran}, then 
\begin{equation}\label{superan}
\left.
\begin{split}
&a_{n+1}^{c}\big(x+\xi-(n+1)c,x\big)\\
&=\max\Big\{\phi\big(x+\xi-(n+1)c,x\big), Q_+\big[a_n^{c}\big]\big(x+\xi-nc,x\big)\Big\}\\
&\geq Q_+\big[a_n^{c}\big]\big(x+\xi-nc,x\big)\\
&= U\big[a_n^{c}(\cdot+\xi-nc, \cdot)\big](x)\;\; \mbox{ for all $(\xi,x)\in\R^2$.}
\end{split}
\right\}
\end{equation}
 We will make use of this observation to prove that $c_+$ is the rightward spreading speed for the recursion 
\[
u_{n+1}=U[u_n],\; u_0\in\mathcal{C},\; n=0,1,2,....
\]

\begin{pro}\label{r-speed}
Let $c_+$ be given in \eqref{defic}.  Suppose $u_0\in\mathcal{C}$ has left supporting point $g_0=-\infty$ and right supporting point $h_0<\infty$, and
\begin{equation}\label{assinit}
\liminf_{x\to-\infty}\big(p(0,x)-u_0(x)\big)>0.
\end{equation}
If for every $C\in\R$, 
\begin{equation}\label{assinitcon}
\lim_{n\to\infty} \big|U^n[u_0](x)-p(0,x)\big|=0\,\hbox{ uniformly in }\,x\in (-\infty, C],
\end{equation}
then $c_+>0$ and 
\begin{equation}\label{rightd1}
\lim_{n\to\infty} \sup_{x\geq c_1n} U^n[u_0](x)=0\quad \hbox{for any } \,c_1> c_+,
\end{equation}
\begin{equation}\label{rightd2}
\lim_{n\to\infty} \sup_{x\leq c_2n} \big|U^n[u_0](x)-p(0,x)\big|=0\quad \hbox{for any } \,c_2< c_+.
\end{equation}
Moreover, 
\begin{equation}\label{liminfc}
\liminf_{n\to\infty} \frac{h_+\big(n\omega;u_0\big)}{n}\geq c_+.
\end{equation}
\end{pro}

\begin{proof}
Without loss of generality, we assume that the right supporting point of $u_0$ is $h_0=0$.  
By replacing $u_0$ by $U[u_0]$ if necessary, we can also assume without loss of generality that, $u_0$ satisfies the assumptions \eqref{assinit}, \eqref{assinitcon} and that 
$$u_0(x)<p(0,x)(1-\epsilon)\, \hbox{ for all } x\in\R \hbox{ for some } \epsilon>0.$$ 
Then we choose some continuous function $l:\R\to [0, 1-\epsilon/2)$ such that $l(x)$ is strictly decreasing in $x\in\R$, that $l(-\infty)=1-\epsilon/2$, $l(0)=1-\epsilon$, $l(1)=0$, and that 
\begin{equation}\label{inicomp}
l(x)p(0,x)\geq u_0(x)\,\hbox{ for all }\,x\in\R.
\end{equation}
Set 
$$\phi(\xi,x):=l(\xi)p(0,x) \,\hbox{ for } (\xi,x)\in\R^2.$$ 
It is easy to see that $\phi\in\mathcal{M}$. In what follows, we will make use of the recursions \eqref{recuran} and \eqref{recurhn} starting from $\phi$ to prove all the conclusions.

We first show that $c_+\in (0,+\infty]$. (We will show in  Proposition \ref{r-speed-finite} that $c_+<+\infty$.) Let $\big\{\big(a_n^{0}(\xi,x), H_n^{0}(\xi)\big)\big\}_{n\in\N}$ be the sequence obtained from the recursions \eqref{recuran} and \eqref{recurhn} with $c=0$.
 By choosing $\xi=0$ in \eqref{superan}, we have
\[
\mbox{ $a_{n+1}^{0}\big(x,x\big) \geq  U\big[a_n^{0}(\cdot, \cdot)\big](x)$ for all $x\in\R$.}
\]
 Then the comparison principle Proposition~\ref{opcompare}~together with \eqref{inicomp} implies that 
 $$a_{n}^{0}(x,x)\geq U^n[u_0](x)\,\hbox{ for all } x\in\R.$$ Furthermore, due to the assumption \eqref{assinitcon}, we see that 
 $$\lim_{n\to\infty}a_{n}^{0}\big(x,x\big)=p(0,x) \, \hbox{ locally uniformly in } x\in\R,$$ 
 whence $a^0(x,x)\equiv p(0,x)$ by Lemma \ref{infian} (i). Since the functions $a^0(\xi,x)$ and $p(0,x)$ are both $L$-periodic in $x\in\R$, we obtain
\[
 a^0(\infty,x)=\lim_{k\to\infty}a^0(x+kL,x)=\lim_{k\to\infty} a^0(x+kL, x+kL)= p(0,x). 
\]
It then follows from Lemma~\ref{proc} that $ c_+>0$.

Next, we prove the convergence property stated in \eqref{rightd1}. If $c_+=+\infty$ then there is nothing to prove. So we assume $c_+<+\infty$. 

Let $\big\{\big(a_n^{c_+}(\xi,x), H_n^{c_+}(\xi)\big)\big\}_{n\in\N}$ be the sequence obtained from the recursions \eqref{recuran} and \eqref{recurhn} with $c=c_+$. By choosing $\xi=0$ and $c=c_+$ in \eqref{superan}, we have 
$$a_{n+1}^{c_+}\big(x-(n+1)c_+,x\big)\geq U\big[a_n^{c_+}(\cdot-nc_+, \cdot)\big](x)\,\hbox{ for all }\, x\in\R,\,n\in\N.$$ 
It then follows from the comparison principle Proposition~\ref{opcompare} and \eqref{inicomp} that
\begin{equation*}
a_{n}^{c_+}\big(x-nc_+,x\big) \geq U^n[u_0](x)  \,\hbox{ for all }\,x\in\R.
\end{equation*}
Furthermore, for any $c_1>c_+$, suppose $\{x_k\}\subset [c_1n,\infty)$ satisfies
\[
\lim_{k\to\infty}U^n[u_0](x_k)=\sup_{x\geq c_1 n}U^n[u_0](x).
\]
Then
since $a_{n}^{c_+}(\xi,x)$ is nonincreasing in $\xi\in\R$ and nondecreasing in $n\in\N$, 
we have
\[
\begin{array}{ll}
\sup_{y\in\R}a^{c_+}\big(nc_1-nc_+,y\big)\!\!& \geq\,\, a^{c_+}(nc_1-nc_+, x_k)\\
\smallskip
&\geq\,\, a^{c_+}(x_k-nc_+, x_k)\\
\smallskip
&\geq\,\, a^{c_+}_n(x_k-nc_+,x_k)\\
\smallskip
&\geq \,\,U^n[u_0](x_k).
\end{array}
\]
Letting $k\to\infty$, we obtain
\begin{equation}
\label{c1>c+}
\sup_{y\in\R}a^{c_+}\big(nc_1-nc_+,y\big)\geq 
\sup_{x\geq c_1n} U^n[u_0](x)\geq 0\,\hbox{ for all }\,n\in\N.
\end{equation}
Since $a^{c_+}(\infty,x)\equiv 0 $ by Lemma \ref{proc} and since $a^{c_+}(nc_1-nc_+,x)$ converges to $a^{c_+}(\infty,x)$ as $n\to\infty$ uniformly in $x\in\R$,  \eqref{rightd1} follows by letting $n\to\infty$ in \eqref{c1>c+}. 

We now prove \eqref{rightd2} and \eqref{liminfc}.
Fix $c_2<c_+$ and let $\big\{\big(a_n^{c_2}(\xi,x), H_n^{c_2}(\xi)\big)\big\}_{n\in\N}$ be the sequence obtained from the recursions \eqref{recuran} and \eqref{recurhn} with $c=c_2$. We claim that there exists $n_0\geq 0$ such that 
\begin{equation}\label{eqrgc}
a_{n+1}^{c_2}(x-c_2,x)=U[a_{n}^{c_2}(\cdot,\cdot)](x)\,\hbox{ for all } \,x\in\R,\,n\geq n_0.
\end{equation} 
Since $c_2<c_+$, it follows from Lemma~\ref{proc} that $a^{c_2}(\infty,x)\equiv p(0,x)$. By Lemma~\ref{comkey}, there exists some $n_0\geq 0$ such that $a_{n_0}^{c_2}(H_2,x)>\phi(-\infty,x)$ for all $x\in\R$, where $H_2$ is a real number such that $\phi(\xi,x)=0$ for all $\xi\geq H_2$ (with our choice of $\phi$, we may take $H_2=1$).
 Then we easily see $a_{n_0}^{c_2}(\xi,x)> \phi(\xi,x)$ for all $\xi\leq H_2,\,x\in\R$, and hence
\[
a_n^{c_2}(\xi,x)\geq  a_{n_0}^{c_2}(\xi,x)> \phi(\xi,x) \mbox{ for all $\xi\in\R,\,x\in\R^2$ and $n\geq n_0$.}
\]
 Thus, by the definition of $a_{n}^{c_2}$, we have
$$a_{n+1}^{c_2}(x+\xi-c_2,x)=U[a_{n}^{c_2}(\cdot+\xi,\cdot)](x)\,\hbox{ for }\,(\xi, x)\in\R,\,n\geq n_0.$$ 
This gives \eqref{eqrgc} by taking $\xi=0$.

By our choice of $\phi$, we can prove by an induction argument that 
$$a_n^{c_2}(x,x)\leq \max_{0\leq k\leq n} U^k[(1-\epsilon/2)p(0,\cdot)](x)\,\hbox{ for all } x\in\R,\,n\in\N. $$
Since $U^k[(1-\epsilon/2)p(0,\cdot)](x)<p(0,x)$ for all $x\in\R$, $k\in\N$ by the strong parabolic maximum principle, 
and these two functions are $L$-periodic in $x$, there exists $\epsilon_n>0$ such that
\[
\max_{0\leq k\leq n} U^k[(1-\epsilon/2)p(0,\cdot)](x)\leq p(0,x)-\epsilon_n \mbox{ for  $x\in\R$ and  $n\in\N$.}
\]
It follows that 
\[
\mbox{$ a_{n_0}^{c_2}(x,x)\leq  p(0,x)-\epsilon_{n_0}$ for all $x\in\R$.}
\]
 Hence, by 
 the assumption \eqref{assinitcon} and the fact that $a_{n_0}^{c_2}(x,x)=0$ for all $x\geq H^{c_2}_{n_0}(0)$, we have
\begin{equation*}
U^{n_1}[u_0](x)\geq a_{n_0}^{c_2}(x,x)\, \hbox{ for all $ x\in\R$ and large integer $n_1$}.
\end{equation*} 
This together with \eqref{eqrgc} and the comparison principle Proposition~\ref{opcompare} implies that 
\[
U^{n_1+n}[u_0](x)\geq U^n[a_{n_0}^{c_2}(\cdot,\cdot)](x)= a_{n+n_0}^{c_2}\big(x-nc_2,x\big)\,\hbox{ for all } \,x\in\R,\,n\in \N,
\]
which yields
\begin{equation}\label{eqrgc2}
U^{n}[u_0](x)\geq a_{n-n_1+n_0}^{c_2}\big(x-(n-n_1)c_2,x\big)\,\hbox{ for all } \,x\in\R,\,n\geq n_1.
\end{equation} 
Since $a_{n}^{c_2}(\xi,x)$ is nonincreasing in $\xi\in\R$, it then follows that 
$$\inf_{x\leq c_2n} U^{n}[u_0](x)\geq \inf_{y\in\R}a_{n-n_1+n_0}^{c_2}\big(n_1c_2,y\big) \mbox{ for } n\geq n_1+1.$$

Since $a^{c_2}(\infty,x)\equiv p(0,x)$ by Lemma \ref{proc}, we have $a^{c_2}(\xi,x)\equiv a^{c_2}(\infty,x)$ by the monotonicity of $a^{c_2}(\xi,x)$ in $\xi$. Therefore, letting $n\to\infty$ in the above inequality, we deduce
 \eqref{rightd2}. 

Moreover, it follows from \eqref{eqrgc2} that 
$$U^{n}[u_0](x)\geq a_{n-n_1+n_0}^{c_2}\big(x-m_nL,x\big)\, \hbox{ for all } \,x\in\R,\,n\geq n_1,$$
where $m_n$  is the positive integer such that 
\[
m_nL\leq  (n-n_1)c_2 < (m_n+1)L.
\]
It follows that
\[
h_+(n\omega;u_0)\geq H_{n-n_1+n_0}^{c_2}(-m_nL)=H_{n-n_1+n_0}^{c_2}(0)+m_nL\, \mbox{ for all $n\geq n_1$.}
\]
 Thus, 
$$\frac{h_+\big(n\omega;u_0\big)}{n}\geq \frac{H_{n-n_1+n_0}^{c_2}(0)}{n}+\frac{m_nL}{n}.$$
Since $H_{n-n_1+n_0}^{c_2}(0)> 0$ (which follows from the monotonicity of $H_{n}^{c_2}(0)$ in $n$ and our choice of $\phi$), passing to the limit $n\to\infty$ yields 
\[
 \liminf_{n\to\infty}\frac{ h_+(n\omega;u_0)}{n}\geq c_2.
\]
 Since $c_2<c_+$ is arbitrary, this implies \eqref{liminfc}. 
The proof of Proposition~\ref{r-speed} is thereby complete. 
\end{proof}

\begin{rem}\label{scspread}
{\rm It is easy to find sufficient conditions for \eqref{assinitcon} to hold. For example, if the nonlinearity $f$ is of type \eqref{logic}, then  \eqref{assinitcon} holds for any $u_0\in\mathcal{C}$ with left supporting point $g_0=-\infty$ and right supporting point $h_0<\infty$. Indeed, in this case, the spreading-vanishing dichotomy in \cite[Theorem~1.2]{ddl} infers that 
\begin{equation*}
\lim_{t\to\infty}h_+(t;u_0)=+\infty\,\hbox{ and }\, \lim_{t\to\infty} u_+(t,x;u_0)=p(t,x) \hbox{ locally uniformly in }  x\in\R.
\end{equation*}
This in particular implies that, for any $C\in\R$, $U^n[u_0](x)$ converges to $p(0,x)$ uniformly in $x\in [C-L,C]$ as $n\to\infty$. By the spatial $L$-periodicity assumption in \eqref{period}, 
$$U^n[u_0](x-kL)= U^n[u_0(\cdot-kL)](x) \,\hbox{ for }\, x\in\R,\,k\in\Z,\,n\in\N.$$ 
By replacing $u_0$ with some $\tilde{u}_0\in \mathcal{C}$ with left supporting point $\tilde{g}_0=-\infty$ such that $\tilde{u}_0(x)$ is nonincreasing in $x$ and $\tilde{u}_0(x)\leq u_0(x)$ for $x\in\R$ if necessary, one can assume without loss of generality that, $u_0(x)$ is nonincreasing in $x\in\R$. It then follows from the order-preserving property of $U$ in (A1) that $$U^n[u_0](x-kL)\geq U^n[u_0](x)\,\hbox{  for }\,x\in [C-L,C],\,k\in\N,\,n\in\N.$$ 
Therefore, $U^n[u_0](x)$ converges to $p(0,x)$ uniformly in $x\leq C$. 
}
\end{rem}

\begin{pro}\label{r-speed-finite}
Let $c_+$ be given in \eqref{defic}. Then $c_+<+\infty$.
\end{pro}
\begin{proof}

By the assumptions on $f$, there exists $K>0$ such that 
\[
f(t,x,u)\leq Ku \mbox{ for } (t,x)\in \R^2, u\in [0,M];\; f(t,x,u)\leq 0 \mbox{ for } u\geq M, \; (t,x)\in\R^2.
\]
It follows that
\[
f(t,x,u)\leq F(u):=\frac{K}{M}u(2M-u) \mbox{ for } (t,x)\in\R^2,\; u\in [0, 2M].
\]

By \cite{dlin}, there exist $c^*>0$ and $\Phi(\xi)\in C^2([0,\infty))$ satisfying
\begin{equation*}
\left\{\baa{l}
\medskip
d\Phi_{xx}-c^*\Phi_x+F(\Phi)=0, \; 0<\Phi(x)<2M \quad\hbox{for } x\in (0,\infty),\vspace{3pt}\\
\Phi(0)=0, \quad \mu \Phi'(0)=c^*,\quad \Phi(\infty)=2M. 
\eaa\right.
\end{equation*} 
Let 
$$\Psi(t,x):=\Phi(c^*t-x+C),\quad  H(t):=c^*t+C  $$
with $C>0$ to be determined later. Then
\[
\left\{
\begin{array}{ll}
\medskip
\Psi_t=d\Psi_{xx}+F(\Psi)\geq d\Psi_{xx}+f(t,x,\Psi) &\mbox{ for } x<H(t), \; t>0,\\
\medskip
\Psi(t,H(t))=0,\; H'(t)=-\mu \Psi_x(t, c^*x+C)& \mbox{ for } t>0,\\
\medskip
\Psi(0,x)=\Phi(-x+C) & \mbox{ for } x\leq H(0)=C.
\end{array}
\right.
\]

Since $f(t,x, u)\leq 0$ for $u\geq M$, a simple comparison argument shows that $p(0,x)\leq M$ for all $x\in\R$.
Let $u_0$ be chosen as in Proposition \ref{r-speed}. Then
\[
u_0(x)<M \,\mbox{ for } x<h_0;\; u_0(x)=0 \,\mbox{ for } x\geq h_0.
\]
Therefore we can fix $C>0$ large enough so that
\[
\Psi(0,x)=\Phi(-x+C)>M>u_0(x)\, \mbox{ for } x\leq h_0.
\]
The comparison principle then yields
\[
h_+(t;u_0)\leq H(t)=c^*t+C,\; u_+(t,x; u_0)\leq \Psi(t,x) \mbox{ for } x\leq h_+(t; u_0),\; t>0.
\]
It then follows that
\[
\limsup_{t\to\infty}\frac{h_+(t;u_0)}{t}\leq c^*.
\]
In particular,
\[
\limsup_{n\to\infty}\frac{h_+(n\omega; u_0)}{n}\leq c^*\omega.
\]
In view of \eqref{liminfc}, we deduce $c_+\leq c^*\omega<+\infty$.
\end{proof}

We are now ready to prove the existence of spreading speed for the problem \eqref{eqfunbd}.

\begin{theo}\label{therspeed} 
Let $c_+^*=c_+/\omega$ where $c_+$ is given in \eqref{defic}. If there exsits $u_0\in\mathcal{H}_+(h_0)$ such that $u_0(x)\leq p(0,x)$ for all $x\in (-\infty,h_0]$ and $u_0$ satisfies  \eqref{assinit} and \eqref{assinitcon}, then
\begin{equation}\label{ctprop}
\lim_{t\to\infty} \sup_{x\leq ct} \big|u_+\big(t,x;u_0\big)-p(t,x)\big|=0\quad \hbox{for any } \,c< c_+^*,
\end{equation}
and 
\begin{equation}\label{cthprop}
\lim_{t\to\infty} \frac{h_+(t;u_0)}{t}=c_+^*.
\end{equation}
\end{theo}

\begin{proof}
In what follows, for any $t\geq 0$, we extend the function $u_+(t,x;u_{0})$ to the whole real line $\R$ by defining $u_+(t,x;u_{0})=0$ for all $x>h_+(t;u_{0})$.
By a slight abuse of notation, we still use $u_+(t,x;u_{0})$ to denote the extended function. 

For each $n\geq 1$, we set 
$$h_n:=h_+(n\omega;u_0) \quad\hbox{and}\quad u_n(x):=u_+(n\omega,x;u_0)=U^n[u_0](x)\,\,\hbox{ for all } \,x\in\R.$$
Clearly
\begin{equation}\label{temporal}
u_+(t+n\omega,x;u_0)=u_+(t,x;u_n)\,\hbox{ and }\,   h_+(t+n\omega;u_0)=h_+(t;u_n)
\end{equation}
for all $n\in\N$, $x\in\R$ and $t\geq 0$.
Furthermore, by the assumption \eqref{assinitcon}, we find that, as $n\to\infty$,
 $h_n\to\infty$ and for any given $C\in\R$, $u_n(x)$ converges to $p(0,x)$ uniformly in $x\in (-\infty,C]$. 

Now we prove  \eqref{ctprop}. It suffices to show that, for any $c< c_+^*$ and any $C\in\R$,
\begin{equation}\label{rspeedcon}
\lim_{t\to\infty} \big|u_+\big(t,x+ct;u_0\big)-p(t,x+ct)\big|=0\, \hbox{ uniformly in }\,x\leq C . 
\end{equation}
Without loss of generality, we can assume that $u_0(x)$ is nonincreasing in $x\in (-\infty, h_0]$. Indeed, we could 
first prove \eqref{rspeedcon} with $u_0$ replaced by some nonincreasing $\tilde{u}_0\in \mathcal{H}_+(h_0)$ satisfying \eqref{assinit} and \eqref{assinitcon} such that $\tilde{u}_0\geq u_0$ in $(-\infty,h_0]$. Then since \eqref{assinitcon} implies
$U^{n_0}[u_0](x)\geq \tilde u_0(x)$ in $\R$ for some large integer $n_0$,   the comparison principle \cite[Proposition~2.14]{ddl} gives
\begin{equation*}
\left.\begin{array}{ll}
u_+(t,x+ct;\tilde{u}_0) &\leq u_+\big(t,x+ct;U^{n_0}[u_0]\big)\\
\smallskip
&=u_+(t+n_0\omega,x+ct;u_0)\\
\smallskip
&\leq p(t+n_0\omega,x+ct)\\
\smallskip
&= p(t,x+ct)   \,\,\hbox{ for all } t>0,\,x\in\R. 
\end{array}\right.
\end{equation*}
This together with \eqref{rspeedcon} holding for $\tilde{u}_0$ implies that, for any $C\in\R$, 
\begin{equation*}
\lim_{t\to\infty} \big|u_+\big(t+n_0\omega,x+ct;u_0\big)-p(t,x+ct)\big|=0\, \hbox{ uniformly in }\,x\leq C, 
\end{equation*}
that is, 
\begin{equation*}
\lim_{t\to\infty} \big|u_+\big(t,x+ct-cn_0\omega;u_0\big)-p(t,x+ct-cn_0\omega)\big|=0\, \hbox{ uniformly in }\,x\leq C. 
\end{equation*}
Thus, \eqref{rspeedcon} holds for the original $u_0$. 

With the assumption that $u_0$ nonincreasing in $x\in (-\infty, h_0]$, it follows from similar analysis as that used in Remark~\ref{scspread} that, to prove \eqref{rspeedcon}, it suffices to show that     
$$\lim_{t\to\infty} \big|u_+\big(t,x+ct;u_0\big)-p(t,x+ct)\big|=0\, \hbox{ locally uniformly in }\,x\in\R. $$
Thus, by \eqref{temporal}, to complete the proof of \eqref{rspeedcon}, we only need to show that for any $c< c_+^*$,
\begin{equation}\label{eqctprop}
\lim_{n\to\infty} \big|u_+(t,x+ct+cn\omega;u_n)-p(t,x+ct+cn\omega)\big|=0 
\end{equation}
locally uniformly in $x\in\R$, and uniformly in $t\in[0,\omega]$. 

Write $cn\omega=x_n+x_n'$ with $x_n\in[0,L]$ and $x_n'\in L\Z$. Due to the spatial $L$-periodicity assumption in \eqref{period} and the $L$-periodicity of the function $p(t,x)$ in $x$, 
\begin{equation*}
\left\{\begin{array}{rcl}
\displaystyle
\medskip
u_+(t,x+ct+cn\omega;u_n)&\!\!\!= \!\!\!&u_+(t,x+ct+x_n;u_n(\cdot+x_n')),\\
\displaystyle
\medskip
p(t,x+ct+cn\omega)&\!\!\! =\!\!\! &p(t,x+ct+x_n)
\end{array} \right. 
\end{equation*}
for any $n\in\N$, $t\in[0,\omega]$, $x\in\R$. Notice that $x_n'\leq c\omega n$ for all $n\in\N$. Since $c< c_+^*$,  and by \eqref{rightd2},
$\lim_{n\to\infty} \sup_{x\leq c_1\omega n} \big|u_n(x)-p(0,x)\big|=0$ for any $c_1\in (c, c_+^*)$, we easily see that
$$ \lim_{n\to\infty} \big|u_n(x+x_n')-p(0,x)\big|=0 \hbox{ locally uniformly in }\, x\in\R. $$
Furthermore, by  \eqref{liminfc}, we have $h(n\omega)-c\omega n\to\infty$ as $n\to\infty$, which clearly implies $h(n\omega)-x_n'\to\infty$
as $n\to\infty$. It then follows from the continuous dependence property stated in \cite[Proposition~2.13]{ddl} that 
$$\lim_{n\to\infty} \big|u_+\big(t,x;u_n(\cdot+x_n')\big) - p(t,x) \big|=0$$
  locally uniformly in $  x\in\R$  and uniformly in $ t\in[0,\omega]$,
which implies \eqref{eqctprop}. The proof of \eqref{ctprop} is thus complete.

Next we prove \eqref{cthprop}. To this end, we first prove the following conclusion:
\begin{equation}\label{ctprop1}
\lim_{t\to\infty} \sup_{x\geq c't}  u_+(t,x;u_0)=0\quad \hbox{for any } \,c'> c_+^*.
\end{equation}

As above, we can assume without loss of generality that $u_0(x)$ is nonincreasing in $x\in (-\infty, h_0]$. For any given $c'> c_+^*$, write $c'n\omega=y_n+y_n'$ with $y_n\in[0,L]$ and $y_n'\in L\Z$. As in the analysis leading to \eqref{eqctprop}, to prove \eqref{ctprop1}, it is sufficient to show that 
\begin{equation}\label{eqctprop1}
\lim_{n\to\infty} u_+(t,x;u_n(\cdot+y_n'))=0 
\end{equation}
 locally uniformly in $ x\in\R$ and uniformly in $t\in[0,\omega]$.

Since $c'> c_+^*$, it follows from \eqref{rightd1} that 
\[
\mbox{
 $ \lim_{n\to\infty} u_n(x+y_n')=0$  locally uniformly in $x\in\R$.}
\]
Denote 
\[
\tilde u_n(x):=u_n(x+y_n')\quad\hbox{and}\quad \tilde h_n=h_n-y_n'.
\]
It is easily seen that  $\tilde{u}_n\in\mathcal{H}_+(\tilde{h}_n)$, and
$\lim_{n\to\infty} \tilde{u}_n(x)=0$ locally uniformly in $x\in\R$.
By the comparison principle we have
\[
0\leq u_+\big(t,x;\tilde{u}_n\big)\leq  v(t,x; \tilde u_n) \,\hbox{ for } t>0,\,x\in\R,
\]
where $v(t,x;\tilde u_n)$ is the solution of the corresponding Cauchy problem, which converges to 0 as $n\to\infty$ uniformly for $(t,x)$ over any bounded set of $[0,+\infty)\times\R$. It then follows that
$$\lim_{n\to\infty} u_+\big(t,x;\tilde{u}_n\big)=0  \hbox{ locally uniformly in } x\in\R \hbox{ and uniformly in } t\in[0,\omega],$$
that is, \eqref{eqctprop1} holds. The proof of \eqref{ctprop1} is finished.

We are now ready to give the proof for \eqref{cthprop}. We first show that 
\begin{equation}\label{liminf}
\liminf_{t\to\infty}  \frac{h_+(t;u_0)}{t} \geq c_+^*.
\end{equation}
 Assume by contradiction that $\liminf_{t\to\infty}  h_+(t;u_0)/t < c_+^*$. Then there would exist some real number $\delta_1>0$ and a sequence $\{t_n\}_{n\in\N}\subset \R$ such that $t_n\to\infty$ as $n\to\infty$, and that $h_+(t_n;u_0)\leq  (c_+^*-\delta_1)t_n$ for all $n\in\N$. It follows from \eqref{ctprop} that 
\begin{equation*}
\liminf_{n\to\infty} \big|u_+\big(t_n,h_+(t_n);u_0\big)-p\big(t_n,h_+(t_n)\big)\big|=0,
\end{equation*}
which is in contradiction with the fact that $u_+\big(t,h_+(t);u_0\big)\equiv 0$. Therefore \eqref{liminf} holds. 

To complete the proof, it remains to show
\begin{equation}\label{limsup}
\limsup_{t\to\infty} \frac{ h_+(t;u_0)}{t} \leq c_+^*.
\end{equation}
Suppose, to the contrary, that $\limsup_{t\to\infty} h_+(t;u_0)/t > c_+^*$. Then we can find 
a real number $\delta_2>0$ such that $\limsup_{t\to\infty} h_+(t;u_0)/t \geq  c_+^*+\delta_2$. Thus, there would exist 
a sequence $\{\tau_n\}_{n\in\N}\subset \R$ such that $\tau_n\to\infty$ as $n\to\infty$, and that 
\begin{equation}\label{contradi}
h_+'(\tau_n;u_0)\geq c_+^*+\frac{\delta_2}{4}, \quad \quad h_+(\tau_n;u_0)\geq  \Big(c_+^*+\frac{\delta_2}{4}\Big)\tau_n\,\,\hbox{ for all }\,n\in\N
\end{equation}
Since $h'_+(t;u_0)=-\mu\partial_xu_+\big(t,h_+(t);u_0\big)$ for all $t\in\R$, it follows from the first inequality of \eqref{contradi} that 
$$ \partial_xu_+\big(\tau_n,h_+(\tau_n);u_0\big) \leq -\frac{c_+^*+\delta_2/4}{\mu}\,\,\hbox{ for all }\,n\in\N.$$
Furthermore, by standard parabolic theory, the function $\partial_{xx}u_+\big(\tau_n,x;u_0\big)$ is  bounded in $(-\infty, h_+(\tau_n;u_0)]$ uniformly in $n$. This implies that there exist two positive constants $C_1$ and $C_2$ independent of $n$ such that 
$$u_+\big(\tau_n,h_+(\tau_n)-C_1;u_0\big) \geq C_2\quad\hbox{for all }\,n\in\N. $$
However, due to the second inequality of \eqref{contradi}, it follows from \eqref{ctprop1} that 
$$\limsup_{n\to\infty}u_+\big(\tau_n,h_+(\tau_n)-C_1;u_0\big) =0, $$
which is a contradiction. Therefore, \eqref{limsup} holds, and  \eqref{cthprop} is proved. 
\end{proof}

Theorem~\ref{therspeed} indicates that $c_+^*$ is the spreading speed for the problem \eqref{eqfunbd} in the rightward direction. In a similar way, we can consider the existence of spreading speed in the leftward direction for the problem  \eqref{eqfunbdn}. As a matter of fact, if we define 
\[
\tilde f(t,x,u)=f(t,-x,u),
\]
 then \eqref{eqfunbdn} reduces to a problem of the form \eqref{eqfunbd} with reaction function $\tilde f$. Therefore a parallel theory holds.

For clarity, we state the corresponding results precisely below. We define
\begin{equation*}
Q_-[{\phi}](\xi,y):=U[{\phi}(\cdot+\xi-y,\cdot)](y)\,\,\hbox{ for }\, {\phi}\in \mathcal{\tilde{M}},
\end{equation*}
where 
\[
\mathcal{\tilde{M}}:=\Big\{\phi\in C(\R^2): \tilde\phi(\xi,x):=\phi(-\xi,-x) \mbox{ belongs to } \mathcal{M}\Big\}.
\]
Clearly, for any fixed $\xi_0\in\R$, $U[\phi(\xi_0+\cdot,\cdot)](y)$ is well-defined, and
\[
U[\phi(\xi_0+\cdot,\cdot)](y):=\left\{\baa{ll}
u_-(\omega,y; \phi(\xi_0+\cdot,\cdot)),& \hbox{ if }\, y\geq g_-(\omega; \phi(\xi_0+\cdot,\cdot)),\vspace{3pt}\\
0,&\hbox{ if }\, y< g_-(\omega; \phi(\xi_0+\cdot,\cdot)).\eaa\right.
\]

Now, for any given ${\phi}\in\mathcal{\tilde{M}}$ and $c\in\R$, we define the sequence $\big\{(b_n^c,G_n^c)\big\}_{n\in\N}$ by the following recursions
\begin{equation*}
b_{n}^c(\xi,x)=\max\Big\{{\phi}(\xi,x), \,Q_-[b_{n-1}^c](\xi-c,x) \Big\} \; \mbox{ with \; $b_{0}^c(\xi,x)={\phi}(\xi,x)$,}
\end{equation*}
and 
\begin{equation*}
G_{n}^c(\xi)=\max\Big\{G_0(\xi),\,g_-\big(\omega;b_{n-1}^c(\cdot+\xi-c,\cdot)\big) \Big\},
\end{equation*}
where   $G_0(\xi)$ is the real number such that ${\phi}(\xi+x,x)=0$ if and only if $\xi\leq G_0(\xi)$. 

\begin{lem} The limits $G^c(\xi)=\lim_{n\to\infty}G_{n}^c(\xi)$ and $b^c(\xi,x)=\lim_{n\to\infty}b_n^c(\xi,x)$ exist. Moreover, either $b^c(-\infty,x)\equiv p(0,x)$  or $b^c(-\infty,x)\equiv 0$. 
\end{lem}

Let $c_-$ be defined by
\begin{equation}\label{deficl}
c_-=\sup\Big\{c\in\R:\, b^c(-\infty,x)\equiv p(0,x)\Big\}.
\end{equation}
Then we have the following two results.

\begin{pro}\label{disspeedl}
Let $c_-$ be given in \eqref{deficl}. Then $c_-\in (0,\infty)$ and is independent of the choice of ${\phi}\in\mathcal{\tilde{M}}$ in the definition of the recursion. Suppose $u_0\in\mathcal{C}$ has left supporting point $g_0>-\infty$,
 right supporting point $h_0=\infty$ and
\begin{equation}\label{assinit1}
\liminf_{x\to\infty}\big(p(0,x)-u_0(x)\big)>0.
\end{equation}
If for every $C\in\R$, 
\begin{equation}\label{assinitcon1}
\lim_{n\to\infty} \big|U^n[u_0](x)-p(0,x)\big|=0\,\hbox{ uniformly in }\,x\in [C,\infty),
\end{equation}
then 
\begin{equation*}
\lim_{n\to\infty} \sup_{x\leq -c_1n} U^n[u_0](x)=0\quad \hbox{for any } \,c_1> c_-,
\end{equation*}
and 
\begin{equation*}
\lim_{n\to\infty} \sup_{x\geq -c_2n} \big|U^n[u_0](x)-p(0,x)\big|=0\quad \hbox{for any } \,c_2< c_-.
\end{equation*}
\end{pro}

\begin{theo}\label{therspeedr} 
Let $c_-^*=c_-/\omega$ where $c_-\in (0,\infty)$ is given in \eqref{deficl}. If there exsits $u_0\in\mathcal{H}_-(g_0)$ such that $u_0(x)\leq p(0,x)$ for all $x\in [g_0,\infty)$ and $u_0$ satisfies the assumptions \eqref{assinit1} and \eqref{assinitcon1}, then
\begin{equation*}
\lim_{t\to\infty} \sup_{x\geq -ct} \big|u_-(t,x;u_0)-p(t,x)\big|=0\quad \hbox{for any } \,c< c_-^*,
\end{equation*}
and
\[
\lim_{t\to\infty} \frac{g_-(t;u_0)}{t}=c_-^*.
\]
\end{theo}


\section{Proof of Theorem~\ref{spreadspeed}}\label{sec4}
In this section, we will complete the proof of Theorem~\ref{spreadspeed} by showing that the rightward and leftward spreading speeds of \eqref{eqf} are the same as the spreading speeds determined by \eqref{eqfunbd} and \eqref{eqfunbdn}, respectively. With the preparations in the previous section, we are now able to adapt the approximation technique of Weinberger \cite{w1} to our situation here, similar in spirit to  \cite[Theorem 3.3]{lz2} where the  Cauchy problem was considered.

Let $\eta:\R\to\R$ be a smooth, nonincreasing function such that $\eta(s)>0$ for $s<1$, and
\begin{equation*}
\eta(s)=\left\{\baa{ll}
1,& \hbox{ if }\, s\leq \frac{1}{2},\vspace{3pt}\\
0,&\hbox{ if }\, s\geq 1.\eaa\right.
\end{equation*} 
For any real number $B>0$, we define the map $U_B$ on $\mathcal{C}$ by 
$$U_B[\varphi](x):=U\Big[\eta\big(\frac{|\cdot-x|}{B}\big)\varphi(\cdot)\Big](x)\,\,\,\hbox{ for all }\, \varphi\in \mathcal{C},  $$
where $U:\mathcal{C}\to\mathcal{C}$ is the operator defined in Section~\ref{sec2}.

Thus, for any given $\varphi\in\mathcal{C}$ with left supporting point $g_0$ and right supporting point $h_0$,
\begin{itemize}
\item
 if $x\not\in(g_0-B, h_0+B)$, then $\eta(|\cdot-x|/B)\varphi(\cdot)\equiv 0$ and hence 
$U_B[\varphi](x)=0$; 
\item
if $x\in (g_0-B,  h_0+B)$, then  $U_B[\varphi](x)$ is 
equal to
$u_{B,x}(\omega,x)$, where $u=u_{B,x}(t,y) $, and its left supporting point $g=g_{B,x}(t)$, right supporting point $h=h_{B,x}(t)$ form a triplet $(u_{B,x},g_{B,x},h_{B,x})$ that solves the following free boundary problem \footnote{As before, for any $t>0$, $u_{B,x}(t,y)$ is extended to $y\in\R$ by defining $u_{B,x}(t,y)=0$ for $y> h_{B,x}(t)$ or $y< g_{B,x}(t)$.} 
\begin{equation}\label{eqtruct}\left\{\baa{l}
u_t(t,y)=d u_{yy}(t,y) + f(t,y,u), \quad g(t)<y<h(t),\quad t>0,\vspace{3pt}\\
u(t,g(t))=u(t,h(t))=0,\quad t>0,\vspace{3pt}\\
g'(t)=-\mu  u_y(t, g(t)), \quad
h'(t)=-\mu  u_y(t,h(t)),\quad t>0, \vspace{3pt}\\
g(0)=\max\{g_0,x-B\},\quad h(0)=\min\{h_0,x+B\},\vspace{3pt}\\
u(0,y)=\eta\big(\frac{|y-x|}{B}\big)\varphi(y), \quad   g(0)\leq y\leq h(0).\eaa\right.
\end{equation} 
\end{itemize}

\begin{lem}\label{trunctedp}
 The operator $U_B$ possesses the following properties. 
\begin{itemize}
\item[(i)]
For any $\varphi\in \mathcal{C}$ and $x\in\R$, $U_B[\varphi](x)$ only depends on the values of $\varphi(y)$ for $y\in [x-B,x+B]$.
\item[(ii)] For any $\varphi\in \mathcal{C}$, $U_B[\varphi](x)$ is nondecreasing in $B$ and converges to $U[\varphi](x)$ locally uniformly in $x\in\R$ as $B\to\infty$.
\item[(iii)]  $U_B$ maps $\mathcal{C}$ into itself, and $U_B$ has the properties stated in {\rm (A1)-(A3)}. 
\end{itemize} 
\end{lem}

\begin{proof}
(i) This statement follows directly from the definition of $U_B$. 

(ii) Let $\varphi$ be a given function in $\mathcal{C}$ with left supporting point $g_0$ and right supporting point $h_0$.
We first prove the monotonicity of $U_B[\varphi](x)$ in $B$.  It suffices to show that for any 
$B_2\geq B_1>0$ and any $x\in (g_0-B_1, h_0+B_1)$, there holds $U_{B_1}[\varphi](x)\leq U_{B_2}[\varphi](x)$. Indeed, for any such $x$, 
$$\max\{g_0,x-B_2\}\leq  \max\{g_0,x-B_1\},\,\,  \, \min\{h_0,x+B_2\}\geq  \min\{h_0,x+B_1\},$$
and  due to the monotonicity of $\eta(s)$ in $s\in\R^+$,
$$\eta\Big(\frac{|y-x|}{B_1}\Big)\varphi(y)\leq \eta\Big(\frac{|y-x|}{B_2}\Big)\varphi(y)\,\hbox{ for  all }\,    y\in\R.$$  
Applying the comparison principle \cite[Proposition~2.10]{ddl} to \eqref{eqtruct}, we obtain
\[
g_{B_1,x}(t)\geq g_{B_2,x}(t),\; h_{B_1,x}(t)\geq h_{B_2,x}(t) \mbox{ for } t>0,
\]
and
\[
 u_{B_1,x}(t,y)\leq u_{B_2,x}(t,y)\; \mbox{ for } t>0,\; y\in[g_{B_1,x}(t), h_{B_1,x}(t)].
\]
In particular,
 $$u_{B_1,x}(\omega,y) \leq u_{B_2,x}(\omega,y) \, \hbox{ for }\,  g_{B_1,x}(\omega) \leq y\leq h_{B_1,x}(\omega).$$
 This clearly implies that $U_{B_1}[\varphi](x)\leq U_{B_2}[\varphi](x)$. 

We now show the convergence of  $U_B[\varphi]$ as $B\to\infty$.  For any given bounded subset $S\subset \R$, when $B$ is sufficiently large, clearly $S\subset (g_0-B, h_0+B)$.  Moreover, as $B\to\infty$,
$$\max\{g_0,x-B\}\to g_0,\,\,\,   \min\{h_0,x+B\}\to h_0  \,\; \hbox{ uniformly in }\,x\in S, $$
 and 
 $$\eta\Big(\frac{|y-x|}{B}\Big)\varphi(y)\to\varphi(y)  \,\; \hbox{ locally uniformly in }\,y\in\R\,\hbox{ and uniformly in }\,x\in S. $$
It then follows from the continuity of the operator $U$ in (A3) that 
$$U\Big[\eta\big(\frac{|\cdot-x|}{B}\big)\varphi(\cdot)\Big](x) \to U[\varphi](x) \,\hbox{ as }\, B\to\infty\, \hbox{ uniformly in }\,x\in S. $$
That is, $U_B[\varphi](x)$ converges to $U[\varphi](x)$ locally uniformly for $x\in\R$ as $B\to\infty$.

(iii) We only prove that $U_B$ maps $\mathcal{C}$ into itself, since  the other properties can be easily checked. It follows easily from the definition of $U_B$ and the properties of $U$ stated in (A1) and (A3) that $U_B[\varphi]\in C(\R)$, and that $0\leq U_B[\varphi](x)\leq p(0,x)$ for all $x\in\R$. 

We now show that, for any $\varphi\in\mathcal{C}$ with left supporting point $g_0=-\infty$ and right supporting point $h_0<\infty$, $U_B[\varphi]$ has the same type of supporting points. Set 
$$h_1=\sup\big\{x\in\R:\, U_B[\varphi](x)>0 \big\}.$$ 
By the continuity of $U_B[\varphi](x)$, we have $U_B[\varphi](h_1)=0$.
If $x_0\leq h_0$, then the right supporting point of $u_{B,x_0}(0,x)$ is $\min\{x_0+B, h_0\}\geq x_0$, and its left supporting point is $x_0-B$. So $h_{B, x_0}(\omega)>x_0,\; g_{B,x_0}(\omega)<x_0-B$. Hence 
\[
U_B[\varphi](x_0)=u_{B, x_0}(\omega, x_0)>0,
\]
which implies
$h_1> h_0$.
For any $x_0\geq h_0+B$, we have $u_{B, x_0}(0,x)\equiv 0$ and hence $U_B[\phi](x_0)=0$. It follows that $h_1\leq h_0+B$.
Summarising the above, we have
\[ h_0<h_1\leq h_0+B,\; 
\mbox{ $U_B[\varphi](x)>0$ for all $x\leq h_0$.}
\]
 We show next  that $U_B[\varphi](x_0)>0$ for all $ h_0<x_0<h_1$. Indeed, for any such $x_0$, by the definition of $h_1$, there exists $x_1 \in (x_0,h_1)$ such that $U_B[\varphi](x_1)>0$. That is, $u_{B,x_1}(\omega,x_1)>0$ and so $h_{B,x_1}(\omega)>x_1>x_0$. 
Furthermore, since $x_1<h_1\leq h_0+B$, the left supporting point of $u_{B,x_1}(0,x)$ is $x_1-B<h_0$. It follows that
$g_{B,x_1}(\omega)<h_0<x_0$. Thus 
\[
\mbox{$g_{B,x_1}(\omega)<x_0<h_{B,x_1}(\omega)$. }
\]
On the other hand, since the function $\eta(s)$ is nonincreasing in $s\in\R^+$, it follows that 
$$\eta\Big(\frac{|x_0-y|}{B}\Big)\varphi(y)\geq \eta\Big(\frac{|x_1-y|}{B}\Big)\varphi(y)\,\hbox{ for }\, y\in [x_1-B,h_0].$$ 
Then applying the comparison principle \cite[Proposition~2.10]{ddl} to \eqref{eqtruct}, we obtain
$$u_{B,x_0}(\omega,y)\geq u_{B,x_1}(\omega,y)\,\hbox{ for }\, g_{B,x_1}(\omega)<y<h_{B,x_1}(\omega).$$ 
This in particular implies that $u_{B,x_0}(\omega,x_0)\geq u_{B,x_1}(\omega,x_0)>0$, that is, $U_B[\varphi](x_0)>0$. 
Thus the left supporting point of $U_B[\varphi]$ is $-\infty$, and its right supporting point is $h_1$.

The analysis for $\varphi\in\mathcal{C}$ with other types of supporting points is similar. The proof of Lemma~\ref{trunctedp} is thereby complete.
\end{proof}

In our analysis below, we need a fixed positive $L$-periodic function $w\in C(\R)$ such that $0<w(x)<p(0,x)$ for all $x\in\R$. We fix a small positive constant $\epsilon$ such that $p(0,x)-\epsilon > w(x)\geq \epsilon$ for all $x\in\R$.
Our first result involving this function $w(x)$ is the lemma below.
\begin{lem}\label{largebn}
There exists $B_1>0$ and $N_1\in\N$ such that 
\begin{equation*}
U_B^n[w](x)\geq p(0,x)-\epsilon\, \hbox{ for all }\,x\in\R,\,B\geq B_1,\,n\geq N_1. 
\end{equation*}  
\end{lem}
\begin{proof}
It follows from the property (A4) that $U^n[w](x)$ converges to $p(0,x)$ as $n\to\infty$ uniformly in $x\in\R$. This implies that there exists some $N_0\in\N$ such that 
$$U^n[w](x)\geq p(0,x)-\epsilon/2\,\hbox{ for all } \,x\in \R,\, n\geq N_0.$$
By Lemma~\ref{trunctedp} (ii), we find some large $B_1>0$ such that 
$$U_B^{N_0+k}[w](x)\geq p(0,x)-\epsilon> w(x)\,\hbox{ for all } \,x\in [0,L],\, B\geq B_1,\, k=0,1,\cdots, N_0-1.$$
For every $n\geq N_0$, there exists $m\in\N$ and $k\in \{0,1,\cdots,N_0-1\}$ such that $n=mN_0+k$, whence it follows from the order-preserving property of $U_B$ that
$$
U_B^{n}[w](x)= U_B^{mN_0+k}[w](x)\geq U_B^{(m-1)N_0}[w](x) \,\hbox{ for all } \,x\in [0,L],\, B\geq B_1.   
 $$
For  $n\geq N_1:=2N_0$, we have $m\geq 2$ and thus 
\[
U_B^{(m-1)N_0}[w](x)\geq U_B^{N_0}[w](x)\geq p(0,x)-\epsilon \,\hbox{ for all } \,x\in [0,L],\, B\geq B_1,   
 \]
where we have used  $U_B^{N_0}[w](x)> w(x)$ to obtain the first inequality.
We thus obtain
\[
U_B^{n}[w](x)\geq p(0,x)-\epsilon \,\hbox{ for all } \,x\in [0,L],\, B\geq B_1,\; n\geq N_1.
\]
Since $U_B$ has property (A2), and $w(x)$ is $L$-periodic, we see that $U_B^n[w](x)$ is $L$-periodic in $x$ for all $n\in\N$. Thus the above inequality holds for all $x\in\R$, and
the proof of Lemma~\ref{largebn} is  complete.
\end{proof}

Let $B_1$ and $N_1$ be given by Lemma \ref{largebn}. 
Fix $B\geq B_1$ and let $\{w_B^n\}_{n\in\N}$ be determined by the following recursion 
$$w_B^n(x)=\max\Big\{ w(x),  U_B\big[w_B^{n-1}\big](x)\Big\}, \quad  w_B^0(x)=w(x). $$
It is easy to see that for each fixed $n\in\N$, $w_B^n(x)$ is positive, continuous and $L$-periodic in $x$, and it is nondecreasing in $n$. Moreover, since $w_B^k(x)\geq w(x)$ for all $k\in\N$, we have, by the proof of Lemma~\ref{largebn},
\begin{equation}
\label{4.2}
w_B^n(x)\geq U_B^n[w](x)>w(x) \mbox{ for }n\geq N_1,\; x\in\R.
\end{equation}
 Then necessarily 
\begin{equation}\label{pp1}
w_B^n(x)=U_B\big[w_B^{n-1}\big](x)\,\hbox{ for all }\, x\in\R,\,n\geq N_1.
\end{equation}

Now, for the same fixed $B\geq B_1$ we define
$$Q_{+,B}[\phi](\xi,x):=U_B[\phi(\cdot+\xi-x,\cdot)](x)\,\,\hbox{ for }\, \phi\in \mathcal{M}.$$
Then $Q_{+,B}$ is a map from $\mathcal{M}$ to $\mathcal{M}$.\footnote{For $\phi\in\mathcal{M}$, since $U_B$ has the properties stated in (A1)-(A3), similar analysis to that in Lemma~\ref{wellde} indicates that $Q_{+,B}[\phi](\xi,x)$ possesses the properties (a)-(c) and (e). To prove (d), one may use the same arguments as those used in the proof of Lemma~\ref{trunctedp} to conclude that, for every $\xi\in\R$, $Q_{+,B}[\phi](\xi+x,x)=0$ if and only if $x\geq \tilde{H}(\xi)$ where $\tilde{H}(\xi)=\sup\{x\in\R: \,U_B[\phi(\cdot+\xi,\cdot)](x)>0\}$.}

 We next fix a function $\phi_0(\xi,x)$ in $\mathcal{M}$ such that $\phi_0(\xi,x)\equiv w(x)$ for all $\xi\leq -1$, and $\phi_0(\xi,x)\equiv 0$ for $\xi\geq  0$. Then for any  $0<c< c_+$, we define 
\begin{equation}\label{4.4}
\tilde{a}_{n+1}^{c}(\xi,x)=\max\Big\{\phi_0(\xi,x), \,Q_{+,B}[\tilde{a}_n^{c}](\xi+c,x) \Big\} \quad\hbox{with}\quad  \tilde{a}_{0}^{c}(\xi,x)=\phi_0(\xi,x).
\end{equation}
It is easily cheked that $\tilde{a}_{n}^{c}(\xi,x)$ is nondecreasing in $n$, nonincreasing in $\xi$ and $c$, and $L$-periodic in $x$. Moreover, we have the following conclusions.

\begin{lem}\label{trunlarsma} 
{\rm (i)} For every $n\in\N$ and $x\in\R$,
\begin{equation}\label{pp2}
\tilde{a}_{n}^{c}(\xi,x)=\left\{\baa{ll}
w_B^n(x),& \hbox{ if }\, \xi\leq -n(B+c)-1,\vspace{3pt}\\
0,&\hbox{ if }\, \xi\geq n(B-c).\eaa\right.
\end{equation} 
$(ii)$ There is some $N_2\in\N$, $B_2>0$ such that 
\begin{equation}\label{pp3}
\tilde{a}_{n+1}^{c}(\xi,x)= Q_{+,B}[\tilde{a}_n^{c}](\xi+c,x) \,\hbox{ for all } x\in\R,\,\xi\in\R,\, B\geq B_2,\, n\geq N_2. 
\end{equation}
\end{lem}

\begin{proof}
We  prove \eqref{pp2} by an induction argument. By our choice of $\phi_0$,  \eqref{pp2} trivially holds in the case of $n=0$. Now suppose that \eqref{pp2} holds for some $n=n_0$. 

By definition,
$$Q_{+,B}[\tilde{a}_{n_0}^{c}](\xi+c,x)=U_B[\tilde{a}_{n_0}^{c}(\cdot+\xi+c-x,\cdot)](x)\,\hbox{ for all }\,\xi\in\R,\,x\in\R.$$
By Lemma~\ref{trunctedp} (i),  $Q_{+,B}[\tilde{a}_{n_0}^{c}](\xi+c,x)$ only depends on the values of $\tilde{a}_{n_0}^{c}(y+\xi+c-x,y)$ with $|x-y|\leq B$.  
For any $\xi\leq -(n_0+1)(B+c)-1$, $x\in\R$ and $y\in\R$ with $|y-x|\leq B$, we have $y+\xi+c-x\leq -n_0(B+c)-1$, whence by the induction assumption $\tilde{a}_{n_0}^{c}(y+\xi+c-x,y)=w_B^{n_0}(y)$, and so
 $Q_{+,B}[\tilde{a}_{n_0}^{c}](\xi+c,x)=U_B[w_B^{n_0}](x)$. This together with the fact $\phi_0(\xi,x)\equiv w(x)$ for such $\xi$ gives
\[
 \tilde{a}_{n_0+1}^{c}(\xi,x)=\max\{w(x), U_B[w^{n_0}_B](x)\}=w_B^{n_0+1}(x).
\]
 Similarly, one concludes that 
$\tilde{a}_{n_0+1}^{c}(\xi,x)\equiv 0$ for $\xi\geq (n_0+1)(B-c)$. Thus \eqref{pp2} also holds for $n=n_0+1$. 
The induction principle then concludes that \eqref{pp2} holds for all $n\in\N$.

Next, we prove \eqref{pp3}.
Since $c< c_+$,  from Lemma~\ref{infian} (i) and Lemma~\ref{proc} we see that, for any 
fixed $\xi\in\R$, $a_{n}^{c}(\xi+x,x)$ converges to $p(0,x)$ locally uniformly in $x\in\R$ as $n\to\infty$, where $\big\{a_{n}^{c}(\xi,x)\big\}_{n\in\N}$ is the sequence obtained from the recursion~\eqref{recuran} with $a_{0}^{c}=\phi_0$. 
Furthermore, by  Lemma~\ref{comkey} and the monotonicity of $a_{n}^c(\xi,x)$ in $n$, $c<c_+$ implies the existence of
 $N_2\in\N$ such that $a_{n+1}^{c}(0,x)> \phi_0(-\infty,x)=w(x)$ for all $x\in\R$, $n\geq N_2$. Then  Lemma~\ref{trunctedp} (ii) implies that there is some $B_2>0$ such that 
\[\mbox{
$\tilde{a}_{n+1}^{c}(0,x)\geq \tilde a_{N_2}^c(0,x)> w(x)$ for all $x\in [0, L]$, $n\geq N_2$, $B\geq B_2$.}
\]
Since both $\tilde{a}_{n+1}^{c}(0,x)$ and $ w(x)$ are $L$-periodic in $x$, the above inequality holds for all $x\in\R$.

 Since $\tilde a_{n+1}(\xi,x)$ is nonincreasing in $\xi$, and since $\phi_0(\xi,x)\leq \phi_0(-\infty,x)\equiv w(x)$, we obtain
$$\tilde{a}_{n+1}^{c}(\xi,x)\geq \tilde a_{n+1}^c(0,x)>  \phi_0(\xi,x)\,\hbox{ for all }\, x\in\R, \,\xi\leq 0,\, n\geq N_2,\,B\geq B_2.$$ 
In view of \eqref{4.4}, this  implies that \eqref{pp3} holds for $\xi\leq 0$.
 Since $\phi_0(\xi,x)\equiv 0$ for $\xi> 0$, we see that \eqref{pp3} also holds for $\xi> 0$.
The proof of Lemma~\ref{trunlarsma} is thereby complete.  
\end{proof}

Correspondingly, we define
$$Q_{-,B}[\tilde{\phi}](\xi,x):=U_B[\tilde{\phi}(\cdot+\xi-x,\cdot)](x)\,\,\hbox{ for }\, \tilde{\phi}\in \tilde{\mathcal{M}},$$
and choose a function $\tilde{\phi}_0(\xi,x)\in\tilde{\mathcal{M}}$ such that $\tilde{\phi}_0(\xi,x)\equiv w(x)$ for $\xi\geq 1$ and $\tilde{\phi}_0(\xi,x)\equiv 0$ for $\xi\leq 0$. For any $0<c'< c_-$, we define 
\begin{equation*}
\tilde{b}_{n+1}^{c'}(\xi,x)=\max\Big\{\tilde{\phi}_0(\xi,x), \,Q_{-,B}[\tilde{b}_n^{c'}](\xi-c',x) \Big\} \quad\hbox{with}\quad  \tilde{b}_{0}^{c'}(\xi,x)=\tilde{\phi}_0(\xi,x).
\end{equation*}
Then $\tilde{b}_{n}^{c'}(\xi,x)$ is nondecreasing in $n$, nondecreasing in $\xi$ and $c$, and $L$-periodic in $x$. Moreover, we have 

 \begin{lem}\label{trunlarsma2} 
{\rm (i)} For every $n\in\N$ and $x\in\R$, 
\begin{equation*}
\tilde{b}_{n}^{c'}(\xi,x)=\left\{\baa{ll}
w_B^n(x),& \hbox{ if }\, \xi\geq n(B+c')+1,\vspace{3pt}\\
0,&\hbox{ if }\, \xi\leq -n(B-c').\eaa\right.
\end{equation*} 
{\rm (ii)} There is $N_3\in\N$, $B_3>0$ such that 
$$\tilde{b}_{n+1}^{c'}(\xi,x)= Q_{-,B}[\tilde{b}_n^{c'}](\xi-c',x) \,\hbox{ for all }\,x\in\R,\,\xi\in\R,\,  B\geq B_3,\, n\geq N_3. $$
\end{lem}

\bigskip

For fixed $c\in (0, c_+)$ and $c'\in (0, c_-)$, let $B_1,\, B_2,\, B_3$ and $N_1,\, N_2,\, N_3$ be given by Lemmas \ref{largebn}, \ref{trunlarsma} and \ref{trunlarsma2}.
Then fix some $B> \max\{B_1,B_2,B_3\}$, some $m> \max\{N_1,N_2,N_3\}$, and choose constants $A$ and $A'$  such that 
$$A\geq \frac{1+m(B+c)+2B}{c},\quad  A'\geq \frac{1+m(B+c')+2B}{c'}.$$
We now define, for every $n\in\N$,
\begin{equation*}
e_{n}(x)=\left\{\baa{ll}
\tilde{a}_{m}^{c}(x-(n+A)c,x),& \hbox{ if }\, x\geq 0,\vspace{3pt}\\
\tilde{b}_{m}^{c'}(x+(n+A')c',x),&\hbox{ if }\, x\leq 0.\eaa\right.
\end{equation*} 

By Lemma~\ref{trunlarsma} (i) and Lemma~\ref{trunlarsma2} (i), it is easy to check that for each $n\in\N$, $e_{n}\in\mathcal{C}$ and that 
\begin{equation}\label{pp4}
e_{n}(x)=\left\{\baa{ll}
w_B^m(x),& \hbox{ if }\,  x\in [-l_{m,n}+1, \tilde l_{m,n}-1],\vspace{3pt}\\
0,&\hbox{ if }\, x\not\in [-l_{m,n}-2m(B+c'), \tilde l_{m,n}+2m(B-c)].\eaa\right.
\end{equation} 
where
\[
l_{m,n}:=(n+A')c'-m(B+c'),\; \tilde l_{m,n}:=(n+A)c-m(B+c).
\]

Furthermore,  the sequence $\{e_n\}_{n\in\N}$ has the following key property.  

\begin{lem}\label{trsubsolution}
$e_{n+1}(x)\leq U_B[e_n](x)$ for all $x\in\R$, $n\in\N$. 
\end{lem}

\begin{proof}
We follow similar lines as the proof of \cite[Lemma 3.10]{lz2}. For the sake of completeness, we include the details here. 

For each $n\in\N$, if $x\in[-l_{m,n}+1+B, \tilde l_{m,n}-1-B]$, then for any $y\in\R$ with $|y-x|\leq B$, we have $y\in[-l_{m,n}+1, \tilde l_{m,n} -1]$, and hence $U_B[e_n](x)=U_B[w_B^m](x)$ by \eqref{pp4}. It then follows from \eqref{pp1} and  \eqref{pp4} that 
$$U_B[e_n](x)=w_B^{m+1}(x),\; w_B^{m}(x)=e_{n+1}(x).$$
By the monotonicity of $w_B^k(x)$ in $k$, we have $w_B^{m+1}(x)\geq w_B^m(x)$ and hence
\[
U_B[e_n](x)\geq e_{n+1}(x).
\]

Now suppose that $x>\tilde l_{m,n}-1-B$. Then for any $y\in\R$ with $|y-x|\leq B$, we have $y>\tilde l_{m,n}-1-2B$. By the choice of $A$, we have  $\tilde l_{m,n}-1-2B>nc>0$, and so $y> 0$. Then by the definition of $e_n$ and the monotonicity of $\tilde a_m^c(\xi,x)$ in $m$, we have
$$ e_n(y)= \tilde{a}_{m}^{c}(y-(n+A)c,y)\geq \tilde{a}_{m-1}^{c}(y-(n+A)c,y),$$
and hence, 
\[
U_B[e_n](x)\geq U_B[\tilde{a}_{m-1}^{c}(\cdot-(n+A)c,\cdot)](x).
\] 
It then follows from Lemma~\ref{trunlarsma} (ii) (by choosing $\xi=x-(n+A)c-c$ and $n=m-1$) and the definition of $e_{n+1}$ that 
$$U_B[e_n](x)\geq U_B[\tilde{a}_{m-1}^{c}(\cdot-(n+A)c,\cdot)](x)=\tilde{a}_{m}^{c}(x-(n+A)c-c,x)=e_{n+1}(x).  $$
Similarly, one can prove that $U_B[e_n](x)\geq e_{n+1}(x)$ if $x< -l_{m,n}+1+B$. The proof of Lemma~\ref{trsubsolution} is thereby complete.
\end{proof}

We are now ready to complete the proof of Theorem~\ref{spreadspeed}.

\begin{proof}[Proof of Theorem~\ref{spreadspeed}]
For any $g_0<h_0$ and $u_0\in\mathcal{H}(g_0,h_0)$ satisfying the assumptions in Theorem~\ref{spreadspeed}, we first extend $u_0$ to the whole real line by defining $u_0(x)=0$ for $x\not\in [g_0, h_0]$. After the extension, clearly $u_0\in \mathcal{C}$. 

For any $n\in\N$, set $u_n(x)=U^n[u_0](x)$. To complete the proof of Theorem~\ref{spreadspeed}, it is sufficient to prove that 
\begin{equation}\label{bothd1}
\lim_{n\to\infty} \sup_{ -c_2n\leq x\leq c_1n} \big|u_n(x)-p(0,x)\big|=0\quad \hbox{for any } \,  c_1\in (0, c_+),\; c_2\in (0,c_-),
\end{equation}
and 
\begin{equation}\label{bothd2}
\lim_{n\to\infty} \sup_{x\in \R\setminus [-c_2'n, c_1'n]}u_n(x)=0 \quad \hbox{for any } \,c_1'> c_+ \hbox{ and } c_2'>c_-.
\end{equation}
Indeed, once \eqref{bothd1} and \eqref{bothd2} are obtained,  similar analysis to that used in the proof of Theorem~\ref{therspeed} would imply  \eqref{spreadu} and \eqref{spreadgf}.

We prove \eqref{bothd2} first. Choose some nondecreasing function $\tilde{u}_0\in\mathcal{C}$ with left supporting point $g_0=-\infty$ and right supporting point $h_0<\infty$ such that 
$\tilde{u}_0(x)\geq u_0(x)$ for all $x\in\R$ and that $\tilde{u}_0$ satisfies \eqref{assinit}. 
Since the operator $U$ is order-preserving in the sense of (A1) and since $u_n(x)$ converges to 
$p(0,x)$ as $n\to\infty$ locally uniformly in $x\in\R$ by our assumption \eqref{assupcon}, it follows that $U^n[\tilde{u}_0]$ also converges to $p(0,x)$ as $n\to\infty$ locally uniformly in $x\in\R$. 
This together with the monotonicity of $\tilde{u}_0$ implies that $U^n[\tilde{u}_0]$ satisfies  \eqref{assinitcon}.
Proposition~\ref{disspeed} then infers
$$\lim_{n\to\infty} \sup_{  x\geq c'_1n} U^n[\tilde{u}_0](x)=0\,\hbox{ for any }\,c'_1>c_+,$$ 
whence $\lim_{n\to\infty} \sup_{  x\geq c'_1n} u_n(x)=0$ by (A1) again. Similarly, by applying Proposition~\ref{disspeedl}, one concludes that 
$$\lim_{n\to\infty} \sup_{  x\leq -c_2'n} u_n(x)=0\,\hbox{ for any }\, c_2'>c_-.$$
Thus \eqref{bothd2} holds.

Next, we prove \eqref{bothd1}. For any given $c_1\in (0,c_+)$ and $c_2\in (0,c_-)$, we fix some $c\in (c_1, c_+)$ and $c'\in (c_2, c_-)$. For any $\epsilon>0$ satisfying $p(0,x)-\epsilon\geq w(x)$, let $B\in\R^+$ and $m\in\N$ large enough such that the conclusions in Lemmas~\ref{largebn}-\ref{trsubsolution} are all valid with $n=m$. Since $\lim_{n\to\infty} |u_n(x)-p(0,x)|=0$ locally uniformly in $x\in\R$ by \eqref{assupcon}, and since $e_0$ is compactly supported in $\R$, there is $l\in\N$ such that $u_l(x)\geq e_0(x)$ for all $x\in\R$. Thus using $U[\varphi]\geq U_B[\varphi]$ for all $\varphi\in\mathcal{C}$, by Proposition~\ref{opcompare} and Lemma~\ref{trsubsolution}, we have
\[
\mbox{ $u_{l+n}(x)\geq e_n(x)$ for all $x\in\R$, $n\in\N$.}
\]
 We may now apply \eqref{pp4} to obtain 
$$u_{l+n}(x)\geq w_B^m(x) \hbox{ for } x\in [-l_{m,n}+1, \tilde l_{m,n}-1],
$$
where
\[
l_{m,n}:= (n+A')c'-m(B+c'),\; \tilde l_{m,n}:= (n+A)c-m(B+c).
\]
By Lemma~\ref{largebn} and our choice of $m$ and $B$, we have 
\[
\mbox{$U_B^{m}[w](x)\geq p(0,x)-\epsilon$ for all $x\in\R$.}
\] 
Since $c_1<c< c_+$ and $c_2<c'<c_-$, there exists $n_1=n_1(c,c_1,c',c_2)$ such that for any $n\geq n_1$, 
\[ [-(l+n)c_2, (l+n)c_1]\subset [-l_{m,n}+1, l_{m,n}-1].
\]
  Then, for $n\geq n_1$ and $x\in [-(l+n)c_2, (l+n)c_1]$, we have
$$u_{l+n}(x)\geq   w_B^m(x)\geq U_B^{m}[w](x)\geq p(0,x)-\epsilon.$$   
It follows that 
\[
\limsup_{n\to\infty}\sup_{ -c_2n\leq x\leq c_1n} u_n(x)\geq p(0,x)-\epsilon.
\]
 Since $\epsilon$ is arbitrary and $u_n(x)\leq p(0,x)$, we thus obtain \eqref{bothd1}. The proof of Theorem~\ref{spreadspeed} is thereby complete.
\end{proof}


\section{Proof of Theorem~\ref{limitmu}}\label{sec5}
In this section we prove that the spreading speeds for the free boundary problem \eqref{eqf} converge to those for the corresponding Cauchy problem \eqref{cauchy} as $\mu\to\infty$. 

By Theorem 1.1, it suffices to show the convergence of the spreading speed for problem \eqref{eqfunbd} in the rightward direction and the same for problem \eqref{eqfunbdn} in the leftward direction, as $\mu\to\infty$. We only consider the former, since  the latter follows from the former by a simple change of variables.

Throughout this section, to indicate the dependence on $\mu$,  for any $u_0\in\mathcal{H}_+(h_0)$, we denote the unique solution of \eqref{eqfunbd} by $\big(u_{+,\mu}(t,x;u_0), h_{+,\mu}(t;u_0)\big)$; and we rewrite 
 $U$, $Q_+$, $a_n^c(\xi,x)$, $H^c_n$, $a^c(\xi,x)$, $H^c$, $c_+$ and $c_+^*$  in Section~\ref{sec3} by $U_{\mu}$, $Q_{+,\mu}$, $a_{n,\mu}^c(\xi,x)$, $H^c_{n,\mu}$, $a_{\mu}^c(\xi,x)$, $H_{\mu}^c$, $c_{+,\mu}$ and $c_{+,\mu}^*$, respectively.

Before starting the proof, let us recall some existing results on the spreading speeds of the Cauchy problem \eqref{cauchy}.
Let
\begin{equation*}
\bar{Q}_+[\phi](\xi,x):=\bar{U}[\phi(\cdot+\xi-x,\cdot)](x) \quad \hbox{for }\,\, \phi\in \mathcal{M},
\end{equation*}
where $\mathcal{M}$ is given in the beginning of Section~\ref{sec3}, and $\bar{U}$ is the Poincar\'{e} map for the Cauchy problem \eqref{cauchy}, that is,
\begin{equation}\label{caupoin}
\bar{U}[\psi](x)=v(\omega,x;\psi)\quad \hbox{for }\,\, \psi\in C(\R).
\end{equation}
It is easily checked that $\bar{Q}_+$ maps $\mathcal{M}$ into $C(\R^2)$ and for any $\phi\in \mathcal{M}$, $\bar{Q}_+[\phi](\xi,x)$ has the properties (a)-(c) and (e). We fix a number $h_0\in\R$ and a function $\phi\in\mathcal{M}$ such that $\phi(\xi,x)\equiv 0$ if and only if $\xi\geq h_0$. For any $c\in\R$, we define the sequence $\{\bar{a}_n^c\}_{n\in\N}$ by the following recursion
\begin{equation}\label{caurecur}
\bar{a}_{n+1}^c(\xi,x)=\max\Big\{\phi(\xi,x), \,\bar{Q}_+[\bar{a}_n^c](\xi+c,x) \Big\}, \quad\quad  \bar{a}_{0}^c(\xi,x)=\phi(\xi,x).
\end{equation}
It follows from the analysis in \cite[Section 3]{w2} that 
\[
\mbox{ $\bar{a}^c(\xi,x):=\lim_{n\to\infty}\bar{a}_n^c(\xi,x)$ exists pointwisely in $\R^2$}
\]
 and that 
\[\mbox{
either $\bar{a}^c(\infty,x)\equiv 0$ or $\bar{a}^c(\infty,x)\equiv p(0,x)$.}
\]
Set
\begin{equation}\label{decaus}
\bar{c}_+=\sup\Big\{c\in\R:\, \bar{a}^c(\infty,x)\equiv p(0,x) \Big\}. 
\end{equation}
Then, \cite[Lemma~3.2]{w2} implies that 
\begin{equation}\label{procacu}
c< \bar{c}_+\,\,\hbox{ if and only if  }\,\,\bar{a}^c_{n_0}(h_0,x)>\phi(-\infty,x) \,\,\hbox{ for some $n_0\in\N$ and all }\,x\in\R. 
\end{equation}
  Moreover, the following lemma is an easy application of \cite[Theorem 2.2]{fyz}.   

\begin{lem}\label{lemscau}
Let $\bar{c}_+^*=\bar{c}_+/\omega$. Then $\bar{c}_+^*$ is the rightward spreading speed for problem~\eqref{cauchy}.
\end{lem}

The next result is the key of this section.

\begin{lem}\label{monoconv}
Let $c_{+,\mu}$ and  $\bar{c}_+$ be given in \eqref{defic} and \eqref{decaus}, respectively. Then $c_{+,\mu}$ is nondecreasing in $\mu>0$ and $\lim_{\mu\to\infty} c_{+,\mu}= \bar{c}_+$.
\end{lem}

\begin{proof}
Let $h_0\in\R$ and $\phi\in\mathcal{M}$ be as in \eqref{caurecur}.
For any real number $c$ and any $\mu>0$,  let the sequence 
$\big\{\big(a^c_{n,\mu}(\xi,x), H^c_{n,\mu}(\xi)\big)\big\}_{n\in\N}$ be obtained from the recursions \eqref{recuran} and \eqref{recurhn}, and $\{\bar{a}^c_{n}(\xi,x)\}_{n\in\N}$ be obtained from \eqref{caurecur}.  

We first claim that, for any $n\in \N$ and any bounded subsets $K_1\subset\R$ and $K_2\subset\R$,  
$a_{n,\mu}^c(\xi+x,x)$ is Lipschitz continuous in $x\in K_2$ uniformly in $\xi\in K_1$ and uniformly in $\mu$ for all large positive $\mu$,
and there holds
\begin{equation}\label{mutoinf0}
\lim_{\mu\to\infty}H^c_{n,\mu}(\xi)=\infty\,\hbox{ uniformly in } \,\xi\in K_1,
\end{equation}
\begin{equation}\label{mutoinf}
\lim_{\mu\to\infty}a_{n,\mu}^c(\xi+x,x)=\bar{a}_{n}^c(\xi+x,x)\,\, \hbox{ uniformly in }\,x\in K_2,\; \xi\in K_1.
\end{equation}

In the case of $n=1$, according to the definitions of $H^c_{n,\mu}$ and $a_{n,\mu}^c$, we have 
$$H^c_{1,\mu}(\xi)=\max\Big\{h_0, h_{+,\mu}\big(\omega;\phi(\cdot+\xi+c,\cdot)\big)\Big\},$$
and
$$a^c_{1,\mu}(\xi+x,x)=\max\Big\{\phi(\xi+x,x), U_{\mu}[\phi(\cdot+\xi+c,\cdot)](x)\Big\}, $$
where $U_{\mu}$ is the operator defined in Section~2.2.
By \cite[Theorem 5.4]{dg2}\footnote{We remark that \cite[Theorem 5.4]{dg2} is concerned with the convergence (as $\mu\to\infty$) of the weak solutions in the sense of \cite[Definition 2.1]{dg2} in high space dimensions.  By some slight modifications of the proof in \cite[Theorem 5.4]{dg2}, we can conclude that such a convergence result is still valid for classical solutions to problem \eqref{eqf} with Lipschitz continuous initial data in $\mathcal{H}(g_0,h_0)$ and for classical solutions to problem \eqref{eqfunbd} with Lipschitz continuous initial data in $\mathcal{H}_+(h_0)$.}, 
we obtain
$$ h_{+,\mu}\big(\omega;\phi(\cdot+\xi+c,\cdot)\big)\to\infty\, \hbox{ as }\,\mu\to\infty\, \hbox{ uniformly in }\, \xi\in K_1,$$ 
and 
$$U_{\mu}[\phi(\cdot+\xi+c,\cdot)](x)\to \bar{U}[\phi(\cdot+\xi+c,\cdot)](x)\,\,\,{ as }\,\,\,\mu\to\infty\,\,\hbox{ in }  C^{1+\alpha}(K_2)$$
 uniformly in $ \xi\in K_1$,
where $\bar{U}$ is  given in \eqref{caupoin}. 
It then follows that
 \eqref{mutoinf0} and \eqref{mutoinf} hold for $n=1$, and $\partial_xU_{\mu}[\phi(\cdot+\xi+c,\cdot)](x)$ is uniformly bounded  in $x\in K_2$, $\xi\in K_1$ for all large $\mu$, say $\mu\geq \mu_0$. Thus, $a_{1,\mu}^c(\xi+x,x)$ is Lipschitz continuous in $x\in K_2$ uniformly in $\xi\in K_1$ and $\mu\geq \mu_0$. This proves our claim  in the case of $n=1$. 

Now, suppose that our claim is valid for some $n=n_0\in\N$; we want to prove that it still holds for $n=n_0+1$. 
Since $a_{n_0,\mu}^c(\xi+x,x)$ is Lipschitz continuous in $x\in K_2$ uniformly in  $\xi\in K_1$ and $\mu\geq \mu_0$, by obvious modifications of \cite[Theorem 5.4]{dg2}, we have
$$\lim_{\mu\to\infty}h_{+,\mu}\big(\omega; a_{n_0,\mu}^c(\cdot+\xi+c,\cdot)\big)=\infty \,\hbox{ uniformly in } \,\xi\in K_1,  $$
and that
$$U_{\mu}[a_{n_0,\mu}^c(\cdot+\xi+c,\cdot)](x)\to \bar{U}[\bar{a}_{n_0}^c(\cdot+\xi+c,\cdot)](x)\,\hbox{ as }\,\mu\to\infty\,\hbox{ in }\,C^{1+\alpha}(K_2)$$ 
uniformly in $\xi\in K_1$. Then the same reasoning  as  in the case of $n=1$ shows that our claim also holds for $n=n_0+1$.

Next, we prove that $c_{+,\mu_1}\leq c_{+,\mu_2}$ whenever $0<\mu_1\leq \mu_2$. Indeed, it follows from the comparison principle \cite[Proposition~2.14]{ddl} that for any $\tilde{\phi}\in\mathcal{M}$, 
$$U_{\mu_1}[\tilde{\phi}(\cdot+\xi-x,\cdot)](x)\leq U_{\mu_2}[\tilde{\phi}(\cdot+\xi-x,\cdot)](x)\,\,\hbox{ for all }\,\xi\in\R,\,x\in\R.$$ 
Since $U_{\mu_1}$ and $U_{\mu_2}$ are order-preserving operators, it follows from an induction argument that 
\begin{equation}
\label{mu1-mu2}
a^c_{n,\mu_1}(\xi,x)\leq a^c_{n,\mu_2}(\xi,x)\, \hbox{ for all }\,\xi\in\R,\,x\in\R.
\end{equation} 
Passing to the limits $n\to\infty$ and $\xi\to\infty$, we obtain 
\[
\mbox{$a^c_{\mu_1}(\infty,x)\leq a^c_{\mu_2}(\infty,x)$ for all $x\in\R$.}
\]
 It follows immediately that 
\[
c_{+,\mu_1}\leq c_{+,\mu_2}.
\]
Moreover, using \eqref{mutoinf} and \eqref{mu1-mu2}, we also obtain
\[
a^c_{n,\mu}(\xi,x)\leq \bar a^c_{n}(\xi,x)\, \hbox{ for all }\,\xi\in\R,\,x\in\R,\; \mu>0.
\]
By passing to the limit $n\to\infty$ and $\xi\to\infty$ it follows  that $c_{+,\mu}\leq \bar{c}_+$. Thus, the limit $\lim_{\mu\to\infty} c_{+,\mu}$ exists and $\lim_{\mu\to\infty} c_{+,\mu}\leq \bar{c}_+$.

To end the proof, we need to show that $\lim_{\mu\to\infty} c_{+,\mu}= \bar{c}_+$. Assume by contradiction that $\lim_{\mu\to\infty} c_{+,\mu}<\bar{c}_+$. Then there exists some $c'\in \R$ such that $c_{+,\mu}<c'<\bar{c}_+$ for all $\mu>0$. It follows from \eqref{procacu} that there is some $n_0>0$ such that 
$$\bar{a}^{c'}_{n_0}(h_0,x)>\phi(-\infty,x)\,\hbox{ for all }\,x\in\R.$$ 
Since the functions $a^{c'}_{n_0,\mu}(h_0,x)$ and $\bar{a}^{c'}_{n_0}(h_0,x)$ are $L$-periodic in $x\in\R$, it follows
from \eqref{mutoinf} that 
$$a^{c'}_{n_0,\mu}(h_0,x)\to \bar{a}^{c'}_{n_0}(h_0,x)\,\hbox{ as }\,\mu\to\infty\, \hbox{ uniformly in } \,x\in\R.$$ 
We then find some $\mu_1>0$ sufficiently large such that 
$$a^{c'}_{n_0,\mu_1}(h_0,x)>\phi(-\infty,x)\,\hbox{ for all }\,x\in\R.$$ 
It follows from Lemma~\ref{comkey} that $c'<c_{+,\mu_1}$, which is in contradiction with the assumption that $c_{+,\mu}<c'$ for all $\mu>0$. The proof of Lemma~\ref{monoconv} is thereby complete.
\end{proof}

\begin{proof}[Proof of Theorem~\ref{limitmu}] 
Let $c^*_{+,\mu}$ be the rightward spreading speed for the free boundary problem \eqref{eqf}. It then follows from Theorem~\ref{therspeed} that $c^*_{+,\mu}=c_{+,\mu}/\omega$. By Lemmas~\ref{lemscau}~and~\ref{monoconv}, 
 $c^*_{+,\mu}$ is nondecreasing in $\mu>0$ and $\lim_{\mu\to\infty}c^*_{+,\mu}=\bar{c}^*_{+}$. 

Correspondingly, we conclude that $c^*_{-,\mu}$ is nondecreasing in $\mu>0$ and $\lim_{\mu\to\infty}c^*_{-,\mu}=\bar{c}^*_{-}$, where $c^*_{-,\mu}$ is the leftward spreading speed for problem \eqref{eqf} and $\bar{c}^*_{-}$ is the leftward spreading speed for problem \eqref{cauchy}. The proof of Theorem~\ref{limitmu} is now complete.
\end{proof}

\end{document}